\documentclass[leqno,a4wide]{amsart}

\usepackage{amssymb,amsmath}
\usepackage{times,bbm,epsfig,float,fancyhdr}
\usepackage{mathtools}
\usepackage{color}
\usepackage[pagebackref=true,colorlinks=true]{hyperref}
\usepackage[pagewise]{lineno}%\linenumbers
\usepackage{tikz}
\usepackage{wrapfig}
\usepackage{algorithm}
\usepackage{algorithmic}

\newcommand{\todo}[1]{$\clubsuit$\ {\tt {\color{red} #1}}\ $\clubsuit$}
\newcommand*{\mailto}[1]{\href{mailto:#1}{\nolinkurl{#1}}}

\usepackage{comment}
\usepackage{graphicx}
\usepackage{subcaption}
\usepackage{amscd}
\usepackage{amssymb}
\usepackage{amsthm}
\usepackage{amsmath}
\usepackage{latexsym}
\usepackage{hyperref}
\usepackage{graphicx} % Required for inserting images
\usepackage{amsmath}
\usepackage{hyperref}
\usepackage{amssymb}
\usepackage{amsfonts}
\usepackage{amsthm}
\usepackage{mathtools}
\usepackage{wasysym}
\usepackage{nicefrac}
\usepackage[shortlabels]{enumitem}
\usepackage{pgfplots}
\mathtoolsset{showonlyrefs=true}		% Show equation numbers only if referred
\usepackage{enumitem}
\usepackage{float}
\usepackage{empheq,mathtools}
\usepackage{amstext}
\usepackage{graphicx}
\usepackage{color}
\usepackage{hyperref}
\usepackage{a4wide}
\usepackage[normalem]{ulem}
\usepackage{bbm}
\usepackage[numbers,sort&compress]{natbib}
\usepackage[a4paper, left=2.5cm, right=2.5cm, top=3.5cm, bottom=2cm]{geometry}
\usepackage{tikz}
\usetikzlibrary{calc,tikzmark,shapes,arrows}

\usepackage{pgfplots}		% Plotting based tikz
\usepackage[english]{babel}

\hypersetup{pdfborder=0 0 0}

\usepackage{tabularx,booktabs,caption}
\newcolumntype{C}[1]{>{\Centering}m{#1}}
\renewcommand\tabularxcolumn[1]{C{#1}}

\newcommand{\R}{\mathbb{R}}

\newtheorem{thm}{Theorem}[section]
\theoremstyle{plain}
\newtheorem{lem}[thm]{Lemma}

\theoremstyle{definition}

\newtheorem{rem}{Remark}[section]

\newcommand{\ZZ}{\mathbb{Z}}			% The set of integers
\newcommand{\NN}{\mathbb{N}}			% The set of natural numbers
\newcommand{\RR}{\mathbb{R}}			% The set of reals
\newcommand{\QQ}{\mathbb{Q}}			% The set of rationals

\newcommand{\isdef}{\coloneqq}			% "Is Defined" symbol
\DeclarePairedDelimiter\abs{\lvert}{\rvert}		% Absolute value (using mathtools)
\DeclarePairedDelimiter\norm{\lVert}{\rVert}	% Norm (using mathtools)
\DeclarePairedDelimiter\floor{\lfloor}{\rfloor}	% Floor (using mathtools)
\DeclarePairedDelimiter\ceil{\lceil}{\rceil}	% Ceiling (using mathtools)
\renewcommand{\complement}{%				% Set complement
	\mathsf{c}%
}
\newcommand{\symdiff}{\mathbin{\triangle}}      % Symmetric difference of sets

\newcommand{\dd}{\mathop{}\!\mathrm{d}}	% Differential
\newcommand{\ee}{\mathrm{e}}			% Euler's constant
\newcommand{\ii}{\mathrm{i}}			% Imaginary unit

\newcommand{\smallo}{o}					% Small o
\newcommand{\bigo}{O}					% Big O

\newcommand{\PP}{\mathbb{P}}
\newcommand\given[1][]{\:#1\vert\:}					% Conditional probability
\newcommand{\EExp}{\operatorname{\mathbb{E}}}               % Expected value
\newcommand{\VVar}{\operatorname{\mathbb{V}\mathrm{ar}}}	% Variance
\newcommand{\CCov}{\operatorname{\mathbb{C}\mathrm{ov}}}		% Covariance
\newcommand{\indicator}[1]{\mathbbm{1}_{#1}}			% Indicator function

\DeclareMathOperator{\maj}{maj}

\usepackage{acro}
\DeclareAcronym{iid}{%
    short=i.i.d\acdot,
    long=independent and identically distributed,
    first-style=short
}
\DeclareAcronym{ie}{%
    short=i.e\acdot,
    long=that is,
    first-style=short
}
\DeclareAcronym{eg}{%
    short=e.g\acdot,
    long=for example,
    first-style=short
}
\newtheorem{theorem}{Theorem}[section]
\newtheorem{lemma}[theorem]{Lemma}
\newtheorem{corollary}{Corollary}[theorem]
\newtheorem{proposition}{Proposition}[section]
\theoremstyle{definition}
\newtheorem{definition}{Definition}
\theoremstyle{remark}
\newtheorem{example}{Example}

\AtBeginEnvironment{example}{%
  \pushQED{\qed}\renewcommand{\qedsymbol}{\ensuremath{\ocircle}}%
}
\AtEndEnvironment{example}{\popQED\endexample}
\AtBeginEnvironment{remark}{%
  \pushQED{\qed}\renewcommand{\qedsymbol}{\ensuremath{\Diamond}}%
}
\AtEndEnvironment{remark}{\popQED\endremark}
\AtBeginEnvironment{example}{%
  \pushQED{\qed}\renewcommand{\qedsymbol}{\ensuremath{\ocircle}}%
}
\AtEndEnvironment{example}{\popQED\endexample}
\AtBeginEnvironment{remark}{%
  \pushQED{\qed}\renewcommand{\qedsymbol}{\ensuremath{\Diamond}}%
}
\AtEndEnvironment{remark}{\popQED\endremark}
\newcommand{\tempcomment}[3][]{
    \noindent{\color{blue}\textup{\texttt{[#1#2: #3]}}}
}
\newcommand{\clara}[2][]{\tempcomment[#1]{C}{#2}}
\newcommand{\anthony}[2][]{\tempcomment[#1]{A}{#2}}
\newcommand{\karim}[2][]{\tempcomment[#1]{K}{#2}}
\newcommand{\hassan}[2][]{\tempcomment[#1]{H}{#2}}
\newcommand{\fatima}[2][]{\tempcomment[#1]{F}{#2}}

\newcommand{\xmodify}[2]{%
    {\color{blue}\sout{#1} #2}
}
\newcommand{\xremove}[1]{%
    {\color{blue}\sout{#1}}
}
\newcommand{\xadd}[1]{%
    {\color{blue}#1}
}
\allowdisplaybreaks
\begin{document}

\title[Derivation, Analysis and Simulation of a Spatio-Temporal Epidemiology Model with Memory]{Derivation, Analysis and Simulation of a Spatio-Temporal Epidemiology Model with Memory}
\author[H. El Bouz]{H. El Bouz}
\address[Hassan El Bouz]
{Department of Mathematics, American University of Beirut, 
Beirut 1107 2020, Lebanon}
\email[]{\mailto{hbe04@mail.aub.edu}}
\author[K. Faraj]{K. Faraj}
\address[Karim Faraj]
{Department of Mathematics, American University of Beirut, 
Beirut 1107 2020, Lebanon}
\email[]{\mailto{kmf16@mail.aub.edu}}
\author[A. Khairallah]{A. Khairallah}
\address[A. Khairallah]
{Department of Mathematics, American University of Beirut, 
Beirut 1107 2020, Lebanon}
\email[]{\mailto{adk05@mail.aub.edu}}
\author[F. Mrou\'e]{F. Mrou\'e}
\address[Fatima Mrou\'e]
{Center for Advanced Mathematical Sciences 
and Department of Mathematics, American University of Beirut, 
Beirut 1107 2020, Lebanon}
\email[]{\mailto{fm47@aub.edu.lb}}
%\maketitle

\thanks{This work was initiated during the Mathematics Summer Research Camp for undergraduate students at the American University of Beirut. The camp is supported by the Dean's office at the Faculty of Arts and Sciences and the Center for Advanced Mathematical Sciences (CAMS). The authors acknowledge the contribution of Miss Clara Riachi during the Summer Research Camp.}

\date{\today}

\begin{abstract}
\textbf{In this paper, we propose an integro-differential model for the spatio-temporal evolution of infectious diseases with asymptomatic transmission. The model consists of a reaction-diffusion system with an integral memory term accounting for the distribution of the incubation period. We first analyze the asymptotic behavior and the properties of the integro-differential model. Then, we prove the local existence of a weak solution of the system by means of the Faedo-Galerkin method and a compactness argument. The model is applied to simulate the geographical evolution of a disease in Lebanon.}

\end{abstract}
\maketitle

\section{Introduction}
\hskip 0.5cm Epidemic models describe the spread of a communicable disease in a population. They are essential to predict the likely outcome of an epidemic \cite{epstein2009modelling}. From a strategic and healthcare management perspective, they help estimate the severity of the epidemic, forecast its time course, and minimize its socioeconomic consequences. A well-known class of models is compartmental models built on differential equations and assuming that the population is perfectly mixed with people moving between compartments such as susceptible (S), infected (I), and recovered (R) \cite{kermack1927contribution,naasell1996quasi,hurley2006basic,jin2007sirs}. These models revealed the threshold nature of epidemics, have been successful in explaining ``herd immunity" and may be extended to include spatial structure.  Since diseases spread differently across various geographical regions, spatial variability is crucial in understanding the impact of public health policies
and interventions in controlling their evolution. Accounting for the movement of people in a spatial context is a persistent challenge in describing spatial variability in epidemics.
  Historically, various approaches have been applied to produce such descriptions, including agent-based
models, network models, stochastic models, and partial differential equation (PDE) models. Agent-based, network, and stochastic models use realistic data of population movement and interaction to simulate human behavior at spatial and temporal levels based on probabilistic assumptions \cite{bonabeau2002agent,ajelli2010comparing,epstein2009modelling,hunter2017taxonomy,hunter2018open,tracy2018agent,kai2020universal,eubank2004modelling,calvetti2020metapopulation}.
However, these models require substantive data collection to set their structural parameters \cite{epstein2007controlling}. Such data are not always available in the early stages of the outbreak or in many developing countries. An alternative approach is the combination of partial differential equations (PDE) with compartmental models that relieve the heavy informational resource. Many authors have applied PDE models to study the evolution of
epidemics in spatial contexts, using diffusion terms to describe the movement of both susceptible and infected individuals or only infectious populations \cite{PierreMagal2020DCDSB,mammeri2020reaction,vaziry2022modelling,allen2008asymptotic,cui2017dynamics,kuniya2017lyapunov,ducrot2014convergence,moritz2023}.\\

In the present work, we propose a spatio-temporal refinement of a compartmental SIR model starting from a discrete time-discrete state space model that is a simplified version of a more complex and realistic model that accounts for social heterogeneities affecting disease transmission \cite{mourad2022stochastic}. The proposed model incorporates the effect of the incubation period of a transmissible disease with an asymptomatic transmission phase through an integral term and then the spatial evolution through a diffusion term. Our objective is to first analyze a reaction-diffusion epidemic model and second to simulate the evolution of an epidemic with asymptomatic transmission in a geographical setting, such as in Lebanon. 
%Our starting point is a discrete time-discrete state space model describing the infection and removal processes leading to a continuous time–continuous state space mathematical model by a limiting process. The discrete-time model with which we start is a simplified version of a more complex and realistic model that takes into account social heterogeneities affecting disease transmission \cite{mourad2022stochastic}.
%with diffusion modeling the infected and susceptible populations and reaction term involving an integral accounting for the incubation period
 
The novelty of the model lies in its incorporation of the incubation period of a viral infection into the removal process. The incubation period is defined as the time from exposure or infection by a microorganism to the onset of the illness and the appearance of symptoms. It can range from a few hours, as in the case of food poisoning, to the order of decades, as in tuberculosis. During the incubation period of many acute infectious diseases, the infected host can be infectious. The incubation period of infectious diseases offers various
insights into clinical and public health practices and is therefore directly related to prevention and control measures. Accordingly, it is involved herein in the removal process as detailed in the next section. Moreover, it allows for the distinction between two classes of infected individuals (infected $I$ and actively infected $I_a$), helping to track, in the future, people who become infected on a particular day and how
long they remain actively infecting others before they are detected by the appearance of symptoms. Several models in the mathematical epidemiology literature handle the incubation period in different ways, depending on assumptions about timing, distributions, and structure. Some are compartmental models with an explicit E compartment for the exposed population (SEIR) during the incubation period, while others use fixed or distributed time delays to model incubation explicitly. There are also integro-differential equation (IDE) models that use integral kernels to model the distribution of incubation times across the population. A typical example of such models is the Kermack--McKendrick model (with distributed delay), which is one of the foundational models that has an integro-differential form. It generalizes the classical SIR model by introducing an infection-age structure, allowing for a distribution of incubation or infectious periods \cite{kermack1932contributions}. A key difference from the present work is that the incubation period is incorporated in the removal process, not in the infectiousness of individuals. The rationale is explained in Section \ref{sec:Modeling}. 

This paper is organized as follows. In Section \ref{sec:Modeling}, we derive the general continuous spatial model, starting from a time-discrete model without space. In Section \ref{sec:analysisTime}, we analyze the continuous temporal model and its asymptotic behavior, and we show the local existence of a weak solution for the spatio-temporal model in Section \ref{sec:analysisSpace}. The numerical results are presented in Section \ref{sec:numerical}. Finally, we conclude with some remarks in Section \ref{sec:Conc}.

\section{Modeling Setup}\label{sec:Modeling}
To present the proposed model, we start from a discrete  model describing the infection and removal processes, and then we obtain, by a heuristic limiting process, a continuous mathematical model. The discrete model is a simplified version of a more complex and realistic model that takes into account social heterogeneities that affect disease transmission \cite{mourad2022stochastic}. Specifically, it incorporates
probability distributions for the size of the family within the
same household, the number of people contacted per day, and their subdivision into known versus unknown
contacts \cite{mourad2022stochastic}. To make the present work self-contained, we detail in the following paragraph the derivation of the simplified discrete model from first principles. Moreover, we note that the present work accounts for the spatial evolution of the disease by introducing a diffusive term into the equations.\\
\subsection{Discrete Time Model}
We start from a fully susceptible population. After the introduction of the microorganism, a susceptible individual can be infected and then removed. By ``removed individuals", we mean those who have been detected and, therefore, isolated. The cumulative number $I(n)$, $I:\mathbb{N}\rightarrow \mathbb{R}^+$, of infected individuals up to a day $n$ is given by: 
\begin{equation} \label{cumulcases}
I(n) = I(n-1) + \Delta I(n,n),
\end{equation}
where $\Delta I(n,n)$ denotes the daily number of new infections that have been infected on day $n$ and $\Delta I(k,n)$, for $k>n$, denotes the remaining number out of $\Delta I(n,n)$ that are still active on day $k$. By convention, we write $\Delta I(0,0) = 0$. We also have the following conservation equation
\begin{equation} \label{conservation}
S(n) + I(n) = M,
\end{equation}
where $M$ denotes the total population, and we assume that the number of births equals the number of deaths, so that the total population is constant.\\
\textit{The infection process} of the susceptible population is considered to evolve as follows.
New infections on day $n$ are introduced by contact with an actively infected individual ($I_a$). Active infected individuals correspond to infected asymptomatic individuals, because they are unaware that they contracted the disease. We further assume that new infections result from contact with actively infected individuals from the time of infection until the day $n-1$. These assumptions are represented by the following relationship
\begin{equation}\label{eq:newContact}
\Delta I(n,n)  =  \displaystyle\beta p_S(n) \sum_{i=1}^{n-1} \Delta I(n-1,n-i)
\end{equation} 
where $\beta$ is the infection probability coefficient and $p_S(n) = \dfrac{S(n-1)}{M}$ is the proportion of susceptible individuals within the total population.\\

On the other hand, for \textit{the removal process} we denote by $\Delta R(n)$ the number of cases removed on day $n$. We have the following relationship
\begin{equation}\label{eq:RemovedX}
\Delta R(n) = \sum_{k=1}^{n-1}\Big[\Delta I(n-1,k)-\Delta I(n,k)\Big].
\end{equation}

We assume that the removal of infected individuals takes place according to the appearance of symptoms which leads to isolation or to the cure of asymptomatic cases so that they are no longer infectious. \\

We denote by $dS$ the discrete random variable of the incubation period, its probability distribution  is obtained from the log-normal distribution. Indeed, the probability density function of incubation periods is usually asymmetric and skewed to the right, and is adequately described by the log-normal distribution for a number of diseases \cite{sartwell1950distribution,brookmeyer2014incubation}.  Then we have the following formula
\begin{equation}\label{eq:RecursiveChange}
\Delta I(n+i,n) = \mu(i+1)\Delta I(n,n)
\end{equation}
where $\displaystyle\mu(i)=1-\sum_{k<i}g(k)=\sum_{k\geq i}g(k)$, where $g(k) = \mathbb{P}[dS = k]$ is the probability that the incubation period is equal to $k$. Notice that $\mu(1) = 1$ and $\displaystyle\lim_{i\rightarrow+\infty} \mu(i) = 0$. %Also by convention, we set $\mu(0)=0$.

Accordingly, the daily new number of removed cases is given by
\begin{equation}\label{eq:RX}
\Delta R(n) = \sum_{i=1}^{n-1}g(n-i)\Delta I(i,i),
\end{equation}
and the total number of cases removed on day $n$ is
\begin{equation}
 R(n) = R(n-1)+\Delta R(n).
\end{equation}

Hence, the current number of active infected cases (i.e., that can infect others) is given by
\begin{equation}
 I_a(n) = I(n) - R(n).
\end{equation}
For a better understanding of the discrete model, in the following Lemma, we express $\Delta I(n,n)$ (given in Equation \eqref{eq:newContact}) solely in terms of $\Delta I(k,k)$ that can be viewed as a function of one time variable $k$. The resulting system can be considered as fitting well into the discrete time series models. Moreover, this way of expression turns out to be helpful in the derivation of a continuous model.\\

\begin{lem}\label{lem:newContactRecursive}
Let $\Delta I(n,n)$ be given as in Equation (\ref{eq:newContact}), then 
\begin{equation}\label{eq:newContactRecursive}
\Delta I(n,n)=\beta p_S(n)\displaystyle\sum_{i=1}^{n-1}\mu(i)\Delta I(n-i,n-i),
\end{equation}
and consequently
\begin{eqnarray}
\Delta I(n,n)&=&\displaystyle\beta\sum_{i=1}^{n-1}\mu(n-i)\Delta I(i,i)
-\dfrac{\beta}{M}\sum_{i=1}^{n-1}\sum_{k=1}^{n-1}\mu(n-i)\Delta I(k,k)\Delta I(i,i).
\end{eqnarray}

\end{lem}
\begin{proof}
The term in the right side of Equation \eqref{eq:newContact} can be written as
\begin{equation}\label{prop1:proof1} 
\begin{array}{ccl}
\displaystyle \sum_{i=1}^{n-1} \Delta I(n-1,n-i)&=&\displaystyle \Delta I(n-1,n-1)+\Delta I((n-2)+1,n-2)\\
&&\displaystyle +\cdots+\Delta I(1+(n-2),1).
\end{array}
\end{equation}
Using Equation \eqref{eq:RecursiveChange} in \eqref{prop1:proof1}, one gets
$$
\begin{array}{rcl}
\displaystyle\sum_{i=1}^{n-1} \Delta I(n-1,n-i)&=&\displaystyle \mu(1)\Delta I(n-1,n-1)+\mu(2)\Delta I(n-2,n-2)\\
&&\displaystyle+\cdots+\mu(n-1)\Delta I(1,1)\\
&=&\displaystyle\sum_{i=1}^{n-1}\mu(i)\Delta I(n-i,n-i).
\end{array}
$$

Using this last equation, one has
\begin{equation}\label{eq:newContactRecursive1}
\nonumber\Delta I(n,n)=\beta p_S(n)\sum_{i=1}^{n-1}\mu(i)\Delta I(n-i,n-i)
\end{equation}
Now, observe that $p_S(n)$ can be written as
\begin{eqnarray*}
p_S(n)&=&1-\dfrac{I(n-1)}{M}\\
&=&1-\dfrac{1}{M}\sum_{k=1}^{n-1}\Delta I(k,k).
\end{eqnarray*}
So, replacing $p_S(n)$ by this last expression in \eqref{eq:newContactRecursive1}, followed by a reindexing, we obtain Equation \eqref{eq:newContactRecursive}.
\end{proof}
\subsection{The Continuous Time  Model}
In the preceding paragraphs, the discrete-time variable was denoted by $n$ and the unit time step was one day. Now, we denote by $t$ the continuous-time variable, and we use the same notation, and the variables $R, S$, and $I$ are considered functions of time $t$. \\

In this section, we formally pass to the limit in the discrete model to derive an equivalent continuous model.\\

Here, the daily new cases $\Delta I(n,n)$ will be replaced, for $n\geq2$, by $I'(t)$. Moreover, we have
\begin{equation}
\displaystyle\sum_{i=1}^{n-1}\Delta I(i,i)\rightarrow I(0)+\int_{0}^{t}I'(s)ds.
\end{equation}
Since $I(1) = \Delta I(1,1)$ and $\Delta I(i,i) = I(i) -I(i-1)$ for $i\geq2$, we can also write
\begin{equation}
\displaystyle\sum_{i=1}^{n-1}\mu(n-i)\Delta I(i,i)\rightarrow \mu(t)I(0)+\int_{0}^{t}\mu(t-s)I'(s)ds.
\end{equation}
 
Equation \eqref{conservation} gives
\begin{equation}
S(t)+I(t) = M.
\end{equation}
Moreover, Equation \eqref{eq:newContactRecursive} gives 
\begin{equation}\label{Sys_SFC1c}
I'(t) = \displaystyle\beta^\star S(t)\Bigg[\mu(t)I(0) +\int_{0}^{t}\mu(t-s)I'(s)ds\Bigg]
\end{equation}
where  $\beta^\star = \dfrac{\beta}{M}$, and  $\displaystyle\mu(t)=\int_t^{+\infty}g(s)ds$, so $\mu'(t)=-g(t)$.\\

Here, the kernel $g$ is a log-normal probability density function supported on $\mathbb{R}_+$, defined by
\begin{equation}\label{eq:g:log-normal}
g(t)=
\begin{cases}
\dfrac{1}{t\sigma\sqrt{2\pi}}
\exp\!\left(-\dfrac{(\ln t-\mu)^2}{2\sigma^2}\right), & t>0,\\[8pt]
0, & t\le 0,
\end{cases}
\end{equation}
where $\mu\in\mathbb{R}$ and $\sigma>0$. In particular,
\[
\int_0^\infty g(t)\,dt = 1.
\]

A simple integration by parts of the terms of the form $\displaystyle\int_{0}^{t}\mu(s)I'(t-s)ds$ leads to terms involving the function $I$ only. Consequently, Equation \eqref{Sys_SFC1c} reduces to an equation whose right side does not involve the derivative of $I(t)$. In other words, we have
\begin{equation}
\displaystyle\int_{0}^{t}\mu(s)I'(t-s)ds = I(t)-\mu(t)I(0)+\int_{0}^{t}\mu'(t-s)I(s)ds.
\end{equation}
Consequently, Equation \eqref{Sys_SFC1c} becomes
\begin{equation}\label{Sys_SFC1d}
\begin{array}{ccl}
I'(t) &=& \displaystyle\beta^\star S(t)\Bigg[I(t)+\int_{0}^{t}\mu'(t-s)I(s)ds\Bigg]\\
&=& \displaystyle\beta^\star S(t)\Bigg[I(t)-\int_{0}^{t}g(t-s)I(s)ds\Bigg].
\end{array}
\end{equation}

For the removed infected cases $t\mapsto R(t)$, Equation \eqref{eq:RX} gives
\begin{equation}
R'(t) = g(t)I(0)+\int_{0}^{t}g(t-s)I'(s)ds = g(0)I(t)+\int_{0}^{t}g'(t-s)I(s)ds.
\end{equation}
However, notice that 
\begin{equation}
\dfrac{d}{dt}\int_{0}^{t}g(t-s)I(s)ds = g(0)I(t)+\int_{0}^{t}g'(t-s)I(s)ds.
\end{equation}
Therefore we have
\begin{equation}\label{eq:RXcontinuous}
R(t) = \int_{0}^{t}g(t-s)I(s)ds,
\end{equation}
where $R(0) = 0$.
In summary, the continuous system is given by:
\begin{equation}\label{Syst1:IDE}\left\{\begin{array}{l}
I'(t) =\displaystyle\beta^\star S(t)\Bigg[I(t)-\int_{0}^{t}g(t-s)I(s)ds\Bigg],\\
S'(t)=\displaystyle-\beta^*S(t)\Bigg[I(t)-\int_{0}^{t}g(t-s)I(s)ds\Bigg]\\
R(t)=\displaystyle\int_{0}^{t}g(t-s)I(s)ds,\\
I(t)+S(t)=M,\\
S(0)=S_0,I(0)=I_0, R(0)=0.
\end{array}\right.
\end{equation} 
\medskip
\noindent\textbf{Properties of the kernel.}
Throughout the paper, we use the kernel
$g : [0,\infty) \to \mathbb{R}_{+}$ defined in \eqref{eq:g:log-normal}. It is simple to check that the kernel $g$ verifies the following properties:
\begin{itemize}
  \item$g(s) \ge 0$ for all $s \ge 0$;
  \item $g \in L^{1}(0,\infty)$ and 
    $\int_{0}^{\infty} g(s)\,ds = 1$;
  \item $g \in C^{1}(0,\infty)$.
\end{itemize}

The active infected cases will be denoted by the function $I_a$ which is defined by 

\begin{align}
    \displaystyle I_a(t) = I(t)-\int_0^t g(t-s)I(s)ds \;\forall t>0.
\end{align}
Solving System \eqref{Syst1:IDE} reduces to solving the integro-differential equation:
\begin{equation}\label{eq:IDE_I}
  \left\{\begin{array}{lcl}
  I'(t) &=&\displaystyle\beta^\star \Big(M-I(t)\Big)\Bigg[I(t)-\int_{0}^{t}g(t-s)I(s)ds\Bigg],\\  
  I(0)&=&I_0,\end{array}\right.
\end{equation}
where $I_0$ is the initial number of infected individuals in the population. Then one can obtain the number of susceptible individuals from
$$S(t)=M-I(t),$$
and the number $R(t)$ of removed individuals is obtained from
$$R(t)=\displaystyle\int_{0}^{t}g(t-s)I(s)ds.$$
Finally, the number of actively infected individuals is deduced from: 
$$I_a(t)=I(t)-R(t).$$
\subsection{The Spatio-Temporal Model}
The propagation of many epidemics often exhibits significant spatial heterogeneity. This observation can be interpreted in several ways; however, the most common strategy involves dividing a large spatial area, such as an entire country, into smaller regions and applying standard ODE-based models to each subdivision. An alternative and increasingly explored approach is to use reaction–diffusion equations to represent the spread of infection as a diffusion-driven process. Reaction–diffusion systems describe the behavior of spatially distributed systems where the components interact both through local reactions and spatial diffusion. These models have been extensively used across disciplines such as physics, chemistry, biology, and epidemiology to investigate how populations or substances evolve and spread over space. In the context of epidemic modeling, reaction–diffusion frameworks are particularly useful for analyzing the spatial propagation of infectious diseases. Although diffusion does not realistically describe the movement of individuals, it is used in epidemics as an averaging process that cannot account for the extreme spatial and temporal heterogeneity of human movement (see, for instance, \cite{bendahmane2010reaction,yin2022reaction,zhi2023influence,zhang2020time,mammeri2020reaction,moritz2023}). Moreover,  it can be used as in \cite{moritz2023} in high-resolution local data, where the authors applied  a reaction-diffusion model to local data in Germany, showing that fast diffusion along major roads was the primary driver of regional outbreaks. %It unrealistically changes the home base of individuals during the time evolution of the disease. 
Furthermore, the spatial movement of the disease agent (microorganism: virus, bacteria) can be described by the diffusion process. We cite, for example, \cite{tu2022dynamics}, where the authors specifically modeled aerosol transmission by treating the viral concentration in the air as a diffusing substance, separate from the movement of the people. Another example is \cite{kammegne2023mathematical}, where the authors assume that individuals move somewhat randomly within their environment (short-range travel, commuting, or daily activities), which causes the ``concentration" of the virus to spread from high-density hotspots to lower-density surrounding areas. \\
%We accordingly propose a model that involves only the movement of infected individuals.\\
%\begin{equation} 
%    \begin{cases}
 %   \displaystyle\frac{\partial S}{\partial t} = - \beta^* S(t,x) \left(I(t,x) - \int_0^tg(t-s)I(s, x)ds \right) ,\quad (t,x)\in [0,T]\times\Omega\\ 
%    \displaystyle\frac{\partial I}{\partial t} = \alpha \nabla^2 I + \beta^* (M - I(t,x)) \left(I(t,x) - \int_0^tg(t-s)I(s, x)ds \right) ,\quad (t,x)\in [0,T]\times\Omega\\ 
%    \displaystyle R(t,x)=  \int_0^tg(t-s)I(s, x)ds ,\quad (t,x)\in [0,T]\times\Omega\\ 
 %   I(0, x)=I_{0}(x)>0,S(0,x)=S_0(x), R(0,x)=R_0(x), x\in \Omega\\
 %  \displaystyle \frac{\partial I}{\partial n} = 0, (t,x)\in[0,T]\times \partial \Omega
  %  \end{cases}
%\end{equation}
%where the diffusion coefficient $\alpha$ is constant.\\
We thus propose a model featuring a diffusive term in the evolution equations of infected and susceptible populations. Hence, we get the following system:
\begin{equation}\label{syst2:PDE}
\begin{cases}
    \displaystyle\dfrac{\partial I}{\partial t} - \nabla^2 I = \beta^\star S \Big[I - \int_0^t g(t-s) I(s,x)\,ds\Big], \ (t,x) \in [0,T]\times\Omega \\ 
    \displaystyle\dfrac{\partial S}{\partial t} - \nabla^2 S = -\beta^\star S [I - \int_0^tg(t-s)I(s,x)ds], \ (t,x) \in [0,T]\times\Omega \\ 
    \frac {\partial{I}} {\partial{n}} = \frac {\partial{S}} {\partial{n}} = 0, (t,x) \in [0,T]\times\partial{\Omega} \\
    S(0,x)=S_0(x), \ I(0,x)=I_0,\ x\in \Omega.
\end{cases}
\end{equation}The homogeneous Neumann boundary condition guarantees that there is no flux of individuals out of the domain boundary. 

\section{Analysis of the Continuous-Time Model}\label{sec:analysisTime}
In this section, we start with the integro-differential system \eqref{Syst1:IDE} corresponding to \eqref{syst2:PDE}, where the functions $x\rightarrow S(t,x)$ and $x\rightarrow I(t,x)$ are assumed to be constants, and we analyze the long-term behavior of the system.
We study the following system of equations:
\begin{equation} \label{syst1:IDE} \quad \quad
    \begin{cases}
    \displaystyle I'(t) = \beta^*(M - I(t)) I_a(t) = \beta^*(M - I(t)) \left(I(t) - \int_0^tg(t-s)I(s)ds \right) \\ 
    \displaystyle S'(t) = -\beta^* S(t) \left(I(t) - \int_0^tg(t-s)I(s)ds\right) \\
    S(0)=S_{0}>0, I(0)=I_{0}>0
    \end{cases}
\end{equation}
where $\beta^* = \frac{\beta}{M}$, and $\beta < 1$ are positive constants, $g$ is a log-normal probability density function defined earlier, and $S_{0}+I_{0}=M$.
\begin{theorem}[Existence and Uniqueness]\label{thm:exist_IDE}
    Let $\beta^*>0$, $0<I_0<M$ and  $0<S_0<M$. Then System \eqref{syst1:IDE} has a unique $C^1$ solution on $[0,\infty)$ such that $S(0)=S_0$, $I(0)=I_0$, $0<S(t)<M$, $0<I(t)<M$, $S(t)$ decreases to $S_\infty>0$ and $I(t)$  increases to $I_\infty<M$.
\end{theorem}
   
\subsection{Proof of Theorem \ref{thm:exist_IDE}}
\subsubsection*{Existence and Uniqueness}
We start by stating the following theorem proved by Burton in \cite{burton2019progressive} that is instrumental in proving the existence and uniqueness of the solution of the integro-differential equation.\\
\begin{theorem}\label{Burton}
Let $T > 0$ be arbitrary and let $E$ be the Banach space of continuous functions on $[0,T]$ equipped with the supremum norm. Assume that there exists a closed convex and bounded set 
\[
G_E \subset (E,\|\cdot\|), 
\quad 
\| G_E \| \le r_E
\]
such that 
\[
P(G_E) \subset G_E,
\]
where $P$ is the mapping defined below, and a function $\phi_0 \in G_E$ with {\color{blue} 
\[
\int_{0}^{t} A(t,s)\,v\!\bigl(t,s,a+ \int_{0}^{s}\phi_0(u)du\bigr)\ ds
\]
bounded on $[0,T]$.} If 
\begin{align}
|g(t,x) - g(t,y)| &\le \alpha\,|x - y|, 
\tag{2.4}\\[6pt]
|f(t,x) - f(t,y)| &\le K\,|x - y|, 
\tag{2.5}\\[6pt]
|v(t,s,x) - v(t,s,y)| &\le L(t,s)\,|x - y|, 
\tag{2.6}
\end{align}
\[
\int_{0}^{t} \bigl|A(t,s)\bigr|\,
L(t,s)\,ds \;\le\; M,
\quad t \in [0,T].
\tag{2.7}
\]

\[
\int_{0}^{t} \bigl|A(t,s)\bigr|\,
L(t,s)\,ds \;\le\; M,
\quad t \in [0,T].
\tag{2.7}
\]
 hold, then the direct mapping \begin{equation}\tag{2.3}
\phi(t)
 \;=\;
 g\!\Bigl(\,
   t,\,
   a \;+\; \int_{0}^{t}\phi(s)\,ds
 \Bigr)
 \;+\;
 f\!\Bigl(\,
   t,\,
   a \;+\; \int_{0}^{t}\phi(s)\,ds
 \Bigr)
 \int_{0}^{t}
   A(t,s)\,
   v\!\Bigl(\,
     t, s,\,
     a \;+\; \int_{0}^{s}\phi(u)\,du
   \Bigr)
 \,ds
\end{equation}
has a unique fixed point, and consequently, the equation \begin{equation}\tag{2.1}
x'(t) 
= g\bigl(t,x(t)\bigr)
  + f\bigl(t,x(t)\bigr)
    \int_{0}^{t}
    A(t,s)\,v\bigl(t,s,x(s)\bigr)\,ds,
\quad
x(0) = a \in \mathbb{R}
\end{equation}
has a unique solution on the interval $[0,T]$. This result then extends to yield a unique solution on $[0,\infty)$.
\end{theorem}

To show the existence and uniqueness of the solution $I(t)$ of the integro-differential equation, observe that it is of the form given in Burton's theorem, by setting 
$$h(t, x(t))=\beta^*(M-x(t))x(t), \ \ f(t, x(t))=-\beta^*(M-x(t)), \ \ A(t,s)=g(t-s),\,\textrm{and }\,v(t, s, x(s))=x(s).$$
We  get
$$I'(t)=h(t,I(t)) + f(t,I(t))\int^t_0 g(t-s)I(s)ds $$

{ We first establish local existence and uniqueness using Burton’s theorem, and then extend the solution to arbitrary time intervals by a continuation argument.

Fix $T>0$, to be chosen later, and define the set $$G_T = \{ I \in C([0,T]) : 0 \le I(t) \le M \text{ for all } t\in[0,T], \, I(0) = I_0, \, I \text{ is non-decreasing} \},$$ and we define the operator $P$ on $G_T$ by:$$(PI)(t) = I_0 + \int_0^t \beta^* (M - I(s)) \left[ I(s) - \int_0^s g(s-\tau) I(\tau) d\tau \right] ds.$$

We prove that for $T$ small enough, $P(G_T)\subset G_T$. Let $I\in G_T$.

The derivative of the image is the integrand:$$(PI)'(t) = \beta^* (M - I(t)) \left[ I(t) - \int_0^t g(t-\tau) I(\tau) d\tau \right].$$ Since $I \in G_T$, we have $I(t) \le M$, so $(M-I) \ge 0$.

Moreover, since $I$ is non-decreasing in $G_T$, we have $I(\tau)\leq I(t)$ for all $\tau\leq t$. Thus, $$\int_0^t g(t-\tau) I(\tau) d\tau \le I(t) \int_0^t g(\tau) d\tau.$$ Since $\int_0^t g(\tau) d\tau \le 1$, it follows that $$ I(t) - \int_0^t g(t-\tau) I(\tau) d\tau \ge 0.$$ 

Therefore, $(PI)'(t) \ge 0$ and $(PI)$ is non-decreasing on $[0,T]$. Also, since $I_0\geq 0$, it follows that $(PI)(t)\geq(PI)(0)=I_0\geq 0$.

We now prove the upper bound. We already have that 
$$0\leq I(t) - \int^t_0g(t-\tau)I(\tau)d\tau\leq I(t) \leq M$$
and $M-I(t)\leq M.$  Using those two facts, we get
$$(PI)'(t) = \beta^* (M - I(t)) \left[ I(t) - \int_0^t g(t-\tau) I(\tau) d\tau \right] \leq \beta^* M^2.$$
Integrating on $[0, t]$ gives 
$$(PI)(t)\leq I_0 + \beta^* M^2 t.$$
Therefore, we choose $T$ such that $T\leq \frac{M-I_0}{\beta^* M^2}$ so that $(PI)(t)\leq M$ for every $t\in[0, T]$ and thus $P(G_T)\subset G_T$.

It remains to verify the Lipschitz assumptions in Burton’s theorem in order to conclude local existence and uniqueness on $[0,T]$.

Now, we verify that the Lipschitz condition for the coefficients hold in the main equation in System \eqref{syst1:IDE}.
\begin{enumerate}[(i)]
   \item 
   $\begin{array}{rl}
   \abs{h(t,I_1)-h(t,I_2)}
   &= \abs{\beta^*(M-I_1)I_1 - \beta^*(M-I_2)I_2} \\
   &= \beta^*\abs{MI_1 - I_1^2 - MI_2 + I_2^2} \\
   &= \beta^*\abs{M(I_1-I_2) - (I_1^2 - I_2^2)} \\ 
   &\leq \beta^*\left( M\abs{I_1-I_2} + \abs{I_1-I_2}\abs{I_1+I_2} \right) \\
   &\leq \beta^*(M+2M)\abs{I_1-I_2}, \qquad \forall I_1,I_2 \in [0,M].
   \end{array}$\\
   So, $\alpha_E = 3\beta^* M$. 

   \item 
   $\abs{f(t,I_1) - f(t, I_2)}
   = \abs{-\beta^*(M-I_1) + \beta^*(M-I_2)}
   = \beta^* \abs{I_1-I_2}$, so $K_E = \beta^*$.

   \item 
   $\abs{v(t, s, I_1) - v(t, s, I_2)}
   = \abs{I_1-I_2} \leq 1 \cdot \abs{I_1-I_2}$, so $L_E(t,s)=1$.

   \item 
   Let $L_E(t,s)=1$. Then
   $\displaystyle\int_0^t|A(t,s)|L_E(t,s)\;ds
   = \int_0^t g(t-s)\,ds
   = \int_0^t g(r)\,dr \le 1$, 
   $t \in [0, T]$, since $\displaystyle\int_0^\infty g(s)\,ds = 1$ and $g \in C(\mathbb{R}^+)$. 
   So, $M_E=1$.
\end{enumerate} 
Hence, by Theorem~\ref{Burton}, there exists a unique $C^1$ solution on $[0, T]$.

We now explain how to extend this local solution to arbitrary time intervals. Suppose that for some $T_1>0$, there exists a unique solution $I$ of our system on $[0,T_1]$. Set $I_1=I(T_1).$
We shall show that the solution can be extended uniquely to a slightly larger interval $[0,T_1+\delta]$.

Define the set
$$G_{T_1,\delta}=\left\{J \in C([T_1,T_1+\delta]) : I_1 \leq J(t)\leq M \text{ for all } t\in[T_1,T_1+\delta], \, J(T_1)=I_1, \, J \text{ is non-decreasing} \right\}.$$

For $J\in G_{T_1,\delta}$, define its extension $J_e$ to $[0,T_1+\delta]$ by
$$
J_e(t)=
\begin{cases}
I(t), & 0\leq t\leq T_1,\\
J(t), & T_1\leq t\leq T_1+\delta.
\end{cases}
$$

We define the continuation operator $P_e$ on $G_{T_1,\delta}$ by
$$
(P_eJ)(t)=I_1+\int_{T_1}^t \beta^*(M-J(s))\left[J(s)-\int_0^s g(s-\tau)J_e(\tau)\,d\tau\right]ds,
\qquad t\in[T_1,T_1+\delta].
$$

The problem is now identical to the previous local case, with $I_0$ replaced by $I_1$. Indeed, following the exact same steps, we obtain that for $\delta \leq \frac{M-I_1}{\beta^*M^2}$,
we have $P_e(G_{T_1,\delta})\subset G_{T_1,\delta}$. The Lipschitz conditions verified above depend only on the coefficients of the equation and remain valid on the translated interval $[T_1,T_1+\delta]$. Hence, Burton's theorem applies again and yields a unique solution on $[T_1,T_1+\delta]$, which extends the previous one.

We now iterate this construction. Define recursively
\[
T_{n+1} = T_n + \delta_n, 
\qquad 
\delta_n = \frac{M - I(T_n)}{\beta^* M^2}.
\]
It remains to show that $T_n \to +\infty$.

Using the fact that $0 \le \int^t_0g(t-\tau)I(\tau)d\tau \le I(t) \le M$, we get
\begin{align*}
    I'(t) &= \beta^*(M - I(t))\int^t_0g(t-\tau)I(\tau)d\tau \\
    &\le \beta^* M (M - I(t)) \\
    & \le \beta^* M S(t).
\end{align*}    
Since $S(t) = M - I(t)$, we also have
\[
S'(t) = -I'(t) \ge -\beta^* M S(t).
\]
Multiplying by $e^{\beta^* M t}$ yields
\[
e^{\beta^*Mt}S'(t)+\beta^*Me^{\beta^*Mt}S(t)\geq 0
\]

\[
\frac{d}{dt}\bigl(e^{\beta^* M t} S(t)\bigr) \ge 0,
\]
so $e^{\beta^* M t} S(t)$ is non-decreasing and
\[
S(t) \ge S(0) e^{-\beta^*Mt}=(M - I_0) e^{-\beta^* M t}.
\]
Evaluating at $t = T_n$, we get
\[
M - I(T_n) \ge (M - I_0) e^{-\beta^* M T_n}.
\]
Hence,
\[
\delta_n = \frac{M - I(T_n)}{\beta^* M^2} \ge \frac{M - I_0}{\beta^* M^2} e^{-\beta^* M T_n}.
\]

Setting $a = \beta^* M$ and $c = \frac{M - I_0}{\beta^* M^2} > 0$, we now have $
T_{n+1} \ge T_n + c e^{-a T_n}.
$

Therefore,
\[
e^{a T_{n+1}} = e^{a T_n} e^{a \delta_n} \ge e^{a T_n}(1 + a \delta_n) \ge e^{a T_n} + ac = e^{a T_0} + n ac \to \infty ,
\]
which implies $T_n \to \infty$.

Thus, the continuation times diverge, and the solution extends uniquely to every interval $[0,T_f]$, and hence to $[0,\infty)$.

This completes the proof of existence and uniqueness.
\begin{rem}
Alternatively, the global existence can be viewed through the lens of a continuation argument by contradiction. Indeed, having shown the existence and uniqueness on $[0,T_0]$, with $T_0= \frac{M-I_0}{\beta M^2}$, let $T_{max}$ be the maximal time for which there is a unique solution on $[0,T_{max})$, so $T_{max}\geq T_0>0$, and the solution is increasing and bounded between $0$ and $M$. Assume $T_{max}<\infty$. Since $I(t)$ is monotone increasing and bounded above by $M$, then $L:=\lim_{t\to T_{max}^-}I(t)$ exists and $L\leq M$.  We identify 2 cases:
\begin{itemize}
    \item[1)] If $L=M$, then %there is $t^*\in(0,T_{max})$ such that $I(t^*)= M$, since $I$ is increasing and less than M, then $I(t) = M$ for all $t\in[t^*,T_{max})$. In this case, 
    we can extend $I(t)$ on $[0,+\infty)$, $\left( \text{namely } I(t)=M \text{ for } t\geq T_{max}\right)$,  and it will be the solution of the IDE since $I'(t) = 0$ and $M-I(t)=0$ for $t\geq T_{max}$.

    \item[2)] If $L < M$, then since $I\in C^1$ and $I_0>0$, and $I$ increasing, then $I(T_{max}) = \max_{t\in[0,T_{max}]}I(t)  = \sup_{t\in[0,T_{max}]}I(t)< M$ and $I'(T_{max}) >0$. In this case, we can use the local existence theorem to extend the solution outside $[0,T_{max}]$, i.e. on $[0,T_{max}+\delta]$ with $\delta>0$. This is a contradiction to the maximality of $T_{max}$. So, $T_{max}=+\infty$.\\ 
    Note that here, $I'(T_{max})$ is defined and finite because the ``memory" integral term involves the $C^\infty$ log-normal kernel $g$ which is integrable over the finite interval $[0,T_{max}]$.
\end{itemize} 
In both cases, we conclude that $T_{max}=+\infty$.\\
This is a qualitative approach showing that there is no finite time singularity.
\end{rem}}
%\begin{rem}Alternatively, the global existence can be viewed through the lens of a continuation argument by contradiction. Indeed, having shown the existence and uniqueness on $[0,T_0]$, with $T_0= \frac{M-I_0}{\beta M^2}$, let $T_{max}$ be the maximal time for which there is a unique solution on $[0,T_{max})$, so $T_{max}\geq T_0>0$, and the solution is increasing and bounded between $0$ and $M$. Assume $T_{max}<\infty$. Since $I(t)$ is monotone increasing and bounded above by $M$, then $L:=\lim_{t\to T_{max}^-}I(t)$ exists and $L\leq M$.  We identify 2 cases:
%\begin{itemize}
%   \item[1)] If $L=M$, then %there is $t^*\in(0,T_{max})$ such that $I(t^*)= M$, since $I$ is increasing and less than M, then $I(t) = M$ for all $t\in[t^*,T_{max})$. In this case, 
    %we can extend $I(t)$ on $[0,+\infty)$, $\left( \text{namely } I(t)=M \text{ for } t\geq T_{max}\right)$,  and it will be the solution of the IDE since $I'(t) = 0$ and $M-I(t)=0$ for $t\geq T_{max}$. This means that $T_{max} = +\infty$.
%    \item[2)] For all $t\in[0,T_{max})$, $I(t) < M$. Assume $T_{max} < +\infty$: then since $I\in C^1$ and $I_0>0$, and I increasing, then $I(T_{max}) = \max_{t\in[0,T_{max}]}I(t)  = \sup_{t\in[0,T_{max}]}I(t)< M$ and $I'(T_{max}) >0$. In this case, we can extend the solution outside $[0,T_{max}]$, i.e. on $[0,T_{max}+\delta]$ with $\delta>0$. This is a contradiction to the maximality of $T_{max}$. So, $T_{max}=+\infty$.
%\end{itemize} 

%This is a qualitative approach showing that there is no finite time singularity.
%\end{rem}}

Having established the existence and uniqueness of a $C^1$ solution for System \eqref{syst1:IDE}, we prove herein some properties of this solution to complete the proof of the theorem.
 
 \begin{proposition}[Monotonicity]\label{prop:Iinc}
     If $0 < I_0 < M$, then $I'(t) > 0$ for all $t$ and, if $I_0 = 0$ or $M$ then $I' = 0$.
 \end{proposition}
 \begin{proof}
         We identify two cases:
    
    If $0<I_0<M$, then $I'(0)>0$, and by continuity, there is an interval where the derivative is positive, so $I(t)$ is increasing on some interval. Let $t^*$ be the first instant of time when $I'(t^*)=0$. If $I(t^*) = M$, then $I(t)=M \quad \forall t$ is the unique solution for the initial condition $I(t^*)=M$, and hence $I_0 = I(0) = M > I_0$. This is a contradiction. Therefore $I(t^*) \neq M$, implying the second term is zero:
    $$I(t^*) = \int_0^{t^*} g(t^*-s)I(s)ds$$ However, using an integration by parts argument:
    \begin{align*}
        I(t^*) &= \int_0^{t^*} g(t^*-s)I(s)ds \\
        &= \left. \int_0^s g(t^*-u)du \;I(s)\right]_0 ^{t^*} - \int_0^{t^*} \int_0^s g(t^*-u)du \; I'(s)ds \\
        &=  I(t^*)\int_0^{t^*} g(t^*-u)du  - \int_0^{t^*} \int_0^s g(t^*-u)du \; I'(s)ds \\
        %&\leq \left. \int_0^s g(t^*-u)du \;I(s)\right]_0^{t^*} \\
        &\leq \int_0^{t^*} g(t^*-s)ds \;I(t^*) \text{ since }I'(s)>0 \text{ for }0<s<t^* \\
        &= \int_0^{t^*} g(s)ds \;I(t^*) \\
        &< I(t^*)\text{ since }\int_0^{t^*} g(s)\;ds<1, 
    \end{align*}
    which is a contradiction. So, $I'(t)>0, \forall t$ for $I_0 \neq 0$. \\

    If $I_0 = 0$ or $I_0 = M$, then the unique solution for this integro-differential equation is respectively $I(t) = 0$ or $I(t)=M$. In both cases, $I'(t) = 0$.\\

 \end{proof}

 \begin{proposition}[Positivity and Boundedness]\label{prop:Ibound}\;\\
  \begin{enumerate}
    \item If $I_0 > 0$, then $I(t) > 0$ for all $t$, and if $I_0 = 0$, then $I(t) = 0$.
    \item If $I_0 < M$, then $I(t) < M$ for all $t$, and if $I_0 = M$, then $I(t) = M$.
\end{enumerate}
     
 \end{proposition}

\begin{proof}\;\\
    \begin{enumerate}
    \item The case $I_0 = 0$ gives us the unique solution $I(t) = 0$. For the case $I_0>0$: Since we start with a positive initial condition $I(0)=I_0>0$, and we have that $I(t)$ is always increasing by Proposition \ref{prop:Iinc}, we conclude that the solution is always 
    strictly positive.\\

    \item The case $I_0 = M$ gives us the unique solution $I(t) = M$. The case $I_0<M$: Assume there exists $t^*$ such that $I(t^*)=M$. Consider the initial value problem with initial condition $I(t^*)=M$, then $J(t)=M$ is a solution for this IVP as well as the solution $I(t)$ which contradicts the uniqueness of the solution. \\
    \end{enumerate}
\end{proof}
\begin{proposition}[Positivity of $I_a$]\label{prop:IaPos}
 If $0 < I_0 \leq M$, then $I_a(t) > 0$ for all $t\geq 0$, and if $I_0 = 0 $, then $I_a = 0$.
\end{proposition}

 \begin{proof}
We identify three cases: \\

    If $I_0 = 0$, then since the unique solution is $I(t) = 0$, $\forall t>0$, one gets $$I_a(t) = I(t)-\int_0^t g(t-s)I(s)ds = 0\;\;\forall t>0.$$
    
Similarly, if $I_0 = M$, then the unique solution is $I(t) = M$, $\forall t>0$. This implies that 
$$I_a(t) = I(t)-\int_0^t g(t-s)I(s)ds = M(1-\int_0^t g(t-s)ds) = M\int_t^\infty g(s)ds>0$$ 

    If $0<I_0<M$, we have $I'(t)=(M-I)I_a>0$ by Proposition \ref{prop:Iinc}. We have $M-I>0$ by Proposition \ref{prop:Ibound}, hence it follows that $$I_a = I(t) - \int_0^t g(t-s)I(s)ds > 0.$$
\end{proof} 
\begin{corollary}[Asymptotic Behavior of $I$]
    $\displaystyle\lim_{t\rightarrow \infty} I(t):=I_\infty \in\RR$ .\\
\end{corollary}
\begin{proof}
    As a consequence of Propositions \ref{prop:Iinc} and \ref{prop:Ibound}, the solution $I(t)$ is an increasing bounded function. Therefore, $\displaystyle\lim_{t\rightarrow \infty} I(t):=I_\infty$ exists.\\
\end{proof}
%Assume the existence of a unique solution $S(t)=S_0 e^{-\int_{0}^{t}\beta I_a(s)ds}$.

\begin{proposition}[Asymptotic Behavior of $I'$]
    $\displaystyle\lim_{t\rightarrow \infty} I'(t)=0$
\end{proposition}

\begin{proof}
    We examine $\displaystyle\lim_{t \to \infty} I'(t)=\beta (M-I_{\infty})\left(I_{\infty}-\lim_{t \to \infty} \int_{0}^{t}g(t-s)I(s)ds\right)$. 

To show that the derivative goes to zero, we will show that $$\lim_{t \to \infty} \int_{0}^{t}g(t-s)I(s)ds = I_{\infty}.$$

Notice that
\begin{align*}
\int_{0}^{t} g(t-s)I(s)ds&=\int_{0}^{t}g(s)I(t-s)ds \ \ \ \ \ \text{by a change of variables}\\
 &= \int_{0}^{\infty}\indicator{[0,t]}g(s)I(t-s)ds\\
 %&\underset{t \to \infty}{\longrightarrow} \int_{0}^{\infty}%g(s)I_{\infty}ds\\&=I_{\infty} 
\end{align*}
%using the fact that $\int_{0}^{\infty} g(t-s)ds=1$.\\

Now, let $f_t(s)=\indicator{[0,t]}g(s)I(t-s)$ and notice then that $0\leq f_{t}(s) \leq Mg(s) \quad \forall s \geq 0$ and $\displaystyle\lim_{t \to \infty}f_{t}(s)=g(s)I_{\infty}:=f$. Since $\displaystyle\lim_{n \to \infty} f_n = f$ and $\lvert f_n \rvert < g$ where $\displaystyle\int_{0}^{\infty} g dx < \infty$, we have by Lebesgue's Dominated Convergence Theorem,  $\displaystyle\lim_{n \to \infty}\int_{0}^{\infty}f_{n}dx=\int_{0}^{\infty}\lim_{n \to \infty}f_{n}dx$. Hence, 
$$\int_{0}^{t} g(t-s)I(s)ds\underset{t \to \infty}{\longrightarrow} \int_{0}^{\infty}g(s)I_{\infty}ds =I_{\infty},$$
using the fact that $\displaystyle\int_{0}^{\infty} g(s)ds=1$.
\end{proof}
Clearly, since $I_a(t)>0$ for all $t$, we have then $S'(t)<0$ for all $t$, hence $S(t)$ is a decreasing function. 
Integrating the ODE of the function $S$ on the interval $[0,t]$, we obtain $\displaystyle S(t) = S_0 e ^{-\gamma \int_0^t I_a(s)ds}$ for all $t\geq 0$.\\
Consequently, the limit as $t\rightarrow+\infty$ of $S(t)$, denoted by $S_\infty$, is given by $\displaystyle S_\infty=S_0 e ^{-\gamma \int_0^\infty I_a(t)dt}$.

We claim that $S_\infty > 0$. This would imply that a certain herd immunity is achieved, since, if $S_\infty$ was to become 0, we would then have that $I_\infty=M-S_\infty=M-0=M$ i.e. all of the population would become infected. We show this claim by the following lemma:
\begin{proposition}[Asymptotic Behaviour of $S$]
       $\displaystyle \int_0^\infty I_a(t)dt$ converges.
\end{proposition}
Note that if it diverges, then we would obtain $S_\infty=S_0e^{-\infty}=0$. \\

\begin{proof}
We have $\displaystyle I_a(t) = I(t) - \int_0^t g(t-s)I(s)ds$.
Choose $T \in \RR$. By a change of variables, let $\displaystyle v(s)=\int_0^s g(T-t)dt=\int_{T-s}^T g(l)dl, \ \ 0 \leq t \leq s \leq T$. \\ 
Using integration by parts, rewrite 
\begin{align*}
\int_0^T g(T-s)I(s)ds &= I(s)v(s)\big]_0^T - \int_0^T I'(s)v(s)ds \\
                        &= I(T)\int_0^T g(l)dl - \int_0^T I'(s)v(s)ds 
\end{align*}
Then $I_a(t)$ becomes
\begin{align*}
I_a(t)&=I(t)-I(t)\int_0^t g(l)dl + \int_0^tI'(s)v(s)ds \\
                    &= I(t)\big[1-\int_0^t g(l)dl \big] + \int_0^tI'(s)v(s)ds \\
                    &= I(t)\big[\int_0^\infty g(t)dt - \int_0^t g(t)dt\big]  + \int_0^tI'(s)v(s)ds\\
                    &= I(t) \int_t^\infty g(l)dl + \int_0^t I'(s) \int_{t-s}^t g(l)dl
\end{align*}
Taking the improper integral of $I_a(t)$, we obtain
\begin{align*}
\int_0^\infty I_a(t)dt &= \int_0^\infty I(t) \int_t^\infty g(l)dl dt + \int_0^\infty \int_0^t I'(s)\int_{t-s}^t g(l)dl ds dt
\end{align*}
We will need to show then that both $\displaystyle (1) \int_0^\infty I(t) \int_t^\infty g(l)dl dt$ and  $\displaystyle (2) \int_0^\infty \int_0^t I'(s)\int_{t-s}^t g(l)dl ds dt$ converge. \\
\begin{enumerate}
    \item $\displaystyle \int_0^\infty I(t) \int_t^\infty g(l)dl dt $ converges.\\
   First, we claim that $\displaystyle\int_0^\infty\int_t^\infty g(l)dl dt=m$, where $m$ is the mean of the log-normal random variable, i.e. 
   $\displaystyle m:= \int_0^\infty l g(l) dl = \exp\left(\mu + \frac{\sigma^2}{2}\right)$. \\
    It is sufficient to prove $\displaystyle\lim_{a \to \infty} \int_0^a \int_t^a g(l)dldt = m$.  Indeed,
    \begin{align*}
        m &= \lim_{a \to \infty} \int_0^a sg(s)ds \\
        &= \lim_{a\to \infty} \left[ s G(s)\Big|_0^a - \int_0^a G(s)ds \right]\\
        &= \lim_{a \to \infty} \left[aG(a) - \int_0^a G(s)ds\right] \\
        &= \lim_{a \to \infty} \left[\int_0^a G(a) - \int_0^a G(s)ds\right] \\
        &= \lim_{a \to \infty} \int_0^a \left[G(a)-G(s)\right]ds \\
        &= \lim_{a \to \infty} \int_0^a \int_s^a g(t)dt ds
    \end{align*}
Consequently, since $I(t)\leq M$ and $g(t)\geq 0$, for all $t$, then $\displaystyle \int_0^\infty I(t) \int_t^\infty g(l)dl dt\leq  \int_0^\infty M \int_t^\infty g(l)dl dt = M\cdot m < +\infty$.

    \item $\displaystyle \int_0^\infty \int_0^t I'(s)\int_{t-s}^t g(l)dl ds dt$ converges. \\

First, by Tonelli's theorem,
\[
\int_0^\infty \int_0^t I'(s) \int_{t-s}^t g(l) \, dl \, ds \, dt
= \int_0^\infty I'(s) \left( \int_s^\infty \int_{t-s}^t g(l) \, dl \, dt \right) ds.
\]

Fix \(s > 0\). Changing the order of integration for the inner integrals,
\[
\int_s^\infty \int_{t-s}^t g(l) \, dl \, dt = \int_0^\infty g(l) \left( \int_{\max(s, l)}^{l+s} dt \right) dl.
\]

Evaluating the inner integral over \(t\),
\[
\int_{\max(s,l)}^{l+s} dt = 
\begin{cases}
l, & s \ge l, \\
s, & l \ge s.
\end{cases}
\]

Hence,
\[
\int_s^\infty \int_{t-s}^t g(l) \, dl \, dt = \int_0^s l g(l) dl + s \int_s^\infty g(l) dl.
\]

By properties of the log-normal distribution,

\[
m := \int_0^\infty l g(l) dl = \exp\left(\mu + \frac{\sigma^2}{2}\right) < \infty,
\]
so for all \(s > 0\),
\[
\int_0^s l g(l) dl \le m.
\]

Moreover, the tail
\[
\bar{G}(s) := \int_s^\infty g(l) dl
\]
decays subexponentially, implying
\[
\lim_{s \to \infty} s \bar{G}(s) = 0,
\]
and \(s \bar{G}(s)\) is bounded on \([0, \infty)\).

Thus,
\[\begin{array}{lcl}
\displaystyle\int_0^\infty I'(s) \left( \int_0^s l g(l) dl + s \int_s^\infty g(l) dl \right) ds
&\le&\displaystyle m \int_0^\infty I'(s) ds + \| s \bar{G}(s) \|_{L^\infty} \int_0^\infty I'(s) ds\\  
&=&\displaystyle m (I_\infty - I_0) + \| s \bar{G}(s) \|_{L^\infty}(I_\infty - I_0)< \infty.
\end{array}
\]
    \end{enumerate}
\end{proof}
This proposition concludes the proof of Theorem \ref{thm:exist_IDE}.
\section{Analysis of the Spatio-Temporal Model}\label{sec:analysisSpace}
In this section, we define the weak solution of System \eqref{syst2:PDE} and prove its local existence. The proof is based on the Faedo-Galerkin method followed by a compactness argument using Aubin-Lions' Lemma.\\
Let \(\Omega \subset \mathbb{R}^d\), \(d=2,3\), be a bounded Lipschitz domain. We define the function spaces:
\[
V := H^1(\Omega), \quad H := L^2(\Omega),
\]
with the standard inner products and norms.
%Recall the system 
%\begin{align*}
%\begin{cases}
%    I_t - \nabla^2 I = \beta S [I - \int_0^t g(t-s) I(s,x) ds], \ x \in \Omega \\ 
%    S_t - \nabla^2 S = -\beta S [I - \int_0^tg(t-s)I(x,s)ds], \ x \in \Omega \\ 
%    \frac {\partial{I}} {\partial{n}} = \frac {\partial{S}} {\partial{n}} = 0, x \in \partial{\Omega} \\
%    S(0,x)=S_0(x), \ I(0,x)=I_0(x)
%\end{cases}
%\end{align*}
Adding the first and the second equations of System \eqref{syst2:PDE}, we get
\begin{align*}
(I+S)_t - \nabla^2 (I+S) = 0.
\end{align*}
So, we let $U(t,x)=(I+S)(t,x)$, then by standard theory, if $U_0\in L^\infty(\Omega)$ there exists a unique weak solution $U\in L^\infty((0,T)\times\Omega)\cap C([0,T];L^2(\Omega))\cap L^2(0,T;H^1(\Omega))$ with $\partial_tU\in L^2(0,T;H^1(\Omega)')$, such that for all $\varphi\in H^1(\Omega)$,
$$\displaystyle\langle U_t,\varphi\rangle_{V',V}+\int_\Omega\nabla U\cdot\nabla \varphi=0,\text{ for almost every } t\in (0,T),$$
with $U(0,x)=U_0(x)$.  Moreover, since the Neumann boundary condition prevents flux through the boundary, the solution $U$ verifies the mass conservation:
$$\displaystyle\int_\Omega U(t,x)\;dx=\int_\Omega U_0(x)\;dx.$$
In addition, as a consequence of the parabolic maximum principle with homogeneous Neumann boundary condition, the solution $U(t,x)$ is bounded by the essential bounds of the initial datum $U_0(x)$. In this case, since $0<U_0<M$, then $0<U(t,x)<M$ for all $(t,x)\in(0,T)\times \Omega$.
%begin{align*}
%\frac {\partial{}} {\partial{t}} \int_\Omega (I+S)dx=0 \implies \int_\Omega (I+S)dx=\int_\Omega (S_0 + I_0)dx
%\end{align*}
%If we want to solve for $I+S$, we get that 
%\begin{align*}
%\begin{cases}
%U = I+S \\ 
%U_t - \nabla^2 U = 0, \ \Omega \\
%\frac {\partial{U}} {\partial{n}} = 0, \ \partial{\Omega} \\
%U(0,x)=S_0(x)+I_0(x) \\
%\int_\Omega U(x,t)dx = \int_\Omega U_0(x)dx, \ \forall t
%\end{cases}
%\end{align*}
%which we can solve similarly to the heat equation. \\
We also rewrite the reaction-diffusion equation involving $I$ by replacing $S=U-I$: \\
\begin{equation}\label{PDE:I:Full}
    I_t - \nabla^2 I = \beta^\star (U-I) \left[I - \int_0^t g(t-s) I(s,x)\,ds\right], \ x \in \Omega.
\end{equation}
Once we show, the existence of solution for this last equation, we obtain the distribution of the susceptible population $S= U-I$.\\
We will use Galerkin approximation, and detailed a priori estimates to prove the existence of a weak solution as defined below.
\subsection*{Weak Formulation}

We consider the PDE:
\begin{equation}\label{eq:mainPDE}
\begin{cases}
\displaystyle I_t - \nabla^2 I = \beta (U - I) \left[ I - \int_0^t g(t-s) I(s,x) \, ds \right], & (t,x) \in (0,T) \times \Omega, \\
I(0,x) = I_0(x) > 0, & x \in \Omega, \\
\partial_n I = 0, & \text{on } (0,T) \times \partial \Omega,
\end{cases}
\end{equation}
where \(\beta >0\), \(U \in L^\infty((0,T)\times\Omega)\), namely \(0<U(t,x)<M\), and \(g\) is the log-normal distribution kernel on \([0,T]\).

\begin{definition}[Weak solution]\label{def:weak}
A function \(I \in L^2(0,T;V)\) with weak time derivative \(I_t \in L^2(0,T;V')\) is a weak solution if
\[
\langle I_t, \phi \rangle_{V',V} + \int_\Omega \nabla I \cdot \nabla \phi \, dx = \beta^\star \int_\Omega (U - I) \left( I - \int_0^t g(t-s) I(s,x)\,ds \right) \phi \, dx
\]
for all \(\phi \in V\) and a.e. \(t \in (0,T)\), with initial condition \(I(0)=I_0\).
\end{definition}
\begin{theorem}\label{thm:Existence of Weak solution}
    Let $I_0\in L^2(\Omega)$. There exists locally in time a weak solution of System \eqref{eq:mainPDE} in the sense of Definition \ref{def:weak}.
\end{theorem}
The proof of the theorem consists of using the Faedo-Galerkin approach with a priori estimates and a compactness argument, as detailed in the following paragraphs.
\subsection{Galerkin Approximation}

Let \(\{w_k\}_{k=1}^\infty\) be a basis consisting of eigenfunctions of \(-\Delta\) with Neumann BC, orthonormal in \(H\) and orthogonal in \(V\).

Define \(V_N = \mathrm{span}\{w_1, \ldots, w_N\}\), and approximate solutions
\[
I_N(t,x) = \sum_{k=1}^N a_k^N(t) w_k(x).
\]

The Galerkin system: for \(k=1,\ldots,N\),
\begin{equation}\label{eq:Galerkin}
\frac{d}{dt} (I_N, w_k) + (\nabla I_N, \nabla w_k) = \beta^\star \left( (U - I_N)(I_N - M_N), w_k \right),
\end{equation}
with initial conditions
\[
(I_N(0), w_k) = (I_0, w_k),
\]
where
\[
M_N(t,x) := \int_0^t g(t-s) I_N(s,x) \, ds.
\]
and \((\cdot,\cdot)\) denotes the $L^2-$inner product.\\
Then System \eqref{eq:Galerkin} can be written as
\begin{equation}\label{eq:Galerk_syst}
\dot{a}^N(t) + K a^N(t) = F(a^N)(t),
\end{equation}
where \(F\) is locally Lipschitz.

\begin{theorem}[Existence of Galerkin solutions]
For each \(N\), there exists a unique solution \(I_N \in C^1([0,T]; V_N)\) to \eqref{eq:Galerkin}.
\end{theorem}

\begin{proof}[Sketch]
The IDE system in \(\mathbb{R}^N\) has a locally Lipschitz right side. The classical IDE theory (Theorem \ref{Burton}) ensures existence and uniqueness of solution of System \eqref{eq:Galerk_syst}. Hence, local existence and uniqueness for System \eqref{eq:Galerkin} follows. The a priori energy bounds (see Section~\ref{sec:apriori}) extend the solution to an interval $[0,T^*]$.
\end{proof}

\subsection{A Priori Estimates}
\label{sec:apriori}
\begin{lem}\label{lem:apriori}
    There exist positive constants $C_1$, $C_2$  and $T^*$ (independent of $N$) such that:
    \[\sup_{t\in[0,T^*]}\|I_N(t)\|_{L^2(\Omega)}^2\leq C_1\text{ and }  \sup_{t\in[0,T^*]}\|\nabla I_N(t)\|_{L^2(\Omega)}^2\leq C_2.\]
\end{lem}
To prove Lemma \ref{lem:apriori}, we first multiply \eqref{eq:Galerkin} by \(a_k(t)\) and sum to get the energy identity:
\[
\frac{1}{2} \frac{d}{dt} \|I_N\|_{L^2}^2 + \|\nabla I_N\|_{L^2}^2 = \beta \int_\Omega (U - I_N)(I_N - M_N) I_N \, dx.
\]

Rewrite the right side of this last equation as:
\[
\int_\Omega (U - I_N)(I_N - M_N) I_N \, dx = \int_\Omega U I_N^2 \, dx - \int_\Omega I_N^3 \, dx - \int_\Omega UI_N M_N \, dx + \int_\Omega I_N^2 M_N \, dx,
\]

Since \(U\) is bounded, we estimate term-by-term.

\subsubsection{Estimate of the cubic term}

\begin{lemma}[Estimate of the cubic term]\label{lemma:cubic-term}
For any \(\varepsilon > 0\), there exists \(C_\varepsilon, C > 0\) such that
\[
\int_\Omega I_N^3 \, dx \leq \varepsilon \|\nabla I_N\|_{L^2}^2 + C_\varepsilon \|I_N\|_{L^2}^6 + C \|I_N\|_{L^2}^3.
\]
\end{lemma}

\begin{proof}
By Sobolev embedding \(H^1 \hookrightarrow L^6\), and interpolation,
\[
\|I_N\|_{L^3} \leq \|I_N\|_{L^2}^{1/2} \|I_N\|_{L^6}^{1/2},
\]
so
\[
\int_\Omega I_N^3 dx = \|I_N\|_{L^3}^3 \leq \|I_N\|_{L^2}^{3/2} \|I_N\|_{L^6}^{3/2}.
\]

Using the embedding,
\[
\|I_N\|_{L^6}^{3/2} \leq C \left(\|\nabla I_N\|_{L^2}^{3/2} + \|I_N\|_{L^2}^{3/2}\right).
\]

Hence
\[
\int_\Omega I_N^3 dx \leq C \|I_N\|_{L^2}^{3/2} \|\nabla I_N\|_{L^2}^{3/2} + C \|I_N\|_{L^2}^3.
\]

Applying Young's inequality with \(p=\frac{4}{3}, q=4\),
\[
\|I_N\|_{L^2}^{3/2} \|\nabla I_N\|_{L^2}^{3/2} \leq \varepsilon \|\nabla I_N\|_{L^2}^2 + C_\varepsilon \|I_N\|_{L^2}^6,
\]
which yields the claim.
\end{proof}

\subsubsection{Estimate of the mixed term}

\begin{lemma}[Estimate of the mixed cubic term]
\label{lemma:mixed-term}
Let \(I_N, M_N \in L^2(\Omega)\) with \(I_N \in H^1(\Omega)\). For any \(\varepsilon\), there exist constants \(C_{\varepsilon}, C > 0\) such that
\[
\int_\Omega I_N^2 M_N \, dx \leq \varepsilon \|\nabla I_N\|_{L^2}^2 + C_\varepsilon \|M_N\|_{L^2}^4 \|I_N\|_{L^2}^2 + C\|M_N\|_{L^2} \|I_N\|_{L^2}^2.
\]
\end{lemma}

\begin{proof}
By Hölder,
\[
\int_\Omega I_N^2 M_N dx \leq \|M_N\|_{L^2} \|I_N\|_{L^4}^2.
\]

By interpolation, we get
\[
\|I_N\|_{L^4} \leq C \|I_N\|_{L^2}^{1/4} \|I_N\|_{L^6}^{3/4}.
\]

Sobolev embedding yields
\[
\|I_N\|_{L^6} \leq C(\|\nabla I_N\|_{L^2} + \|I_N\|_{L^2}).
\]

Thus,
\[
\|I_N\|_{L^4}^2 \leq C \|I_N\|_{L^2}^{1/2} (\|\nabla I_N\|_{L^2} + \|I_N\|_{L^2})^{3/2}.
\]

Expanding,
\[
\int_\Omega I_N^2 M_N dx \leq C \|M_N\|_{L^2} \|I_N\|_{L^2}^{1/2} (\|\nabla I_N\|_{L^2}^{3/2} + \|I_N\|_{L^2}^{3/2}).
\]

Apply Young's inequality with \(p=\frac{4}{3}\), \(q=4\) to
\[
\|M_N\|_{L^2} \|I_N\|_{L^2}^{1/2} \|\nabla I_N\|_{L^2}^{3/2} \leq \varepsilon \|\nabla I_N\|_{L^2}^2 + C_\varepsilon \|M_N\|_{L^2}^4 \|I_N\|_{L^2}^2.
\]

%For the term \(\|M_N\|_{L^2} \|I_N\|_{L^2}^2\), 
%\[
%\|M_N\|_{L^2} \|I_N\|_{L^2}^2 \leq \frac{1}{2}\left(\|M_N\|_{L^2}^2 + \|I_N\|_{L^2}^4\right).
%\]

%use Young with exponents \(p=4, q=4/3\):
%\[
%\|M_N\|_{L^2} \|I_N\|_{L^2}^2 \leq \frac{\|M_N\|_{L^2}^4}{4 \delta^3} + \frac{3 \delta}{4} \|I_N\|_{L^2}^{8/3}.
%\]
%
%To control the nonlinear term \( \|I_N\|_{L^2}^{8/3} \) arising in the energy estimates, we use a nonlinear Young-type inequality: for any \( \eta > 0 \), there exists a constant \( C_\eta > 0 \) such that
%\[
%\|I_N\|_{L^2}^{\frac{8}{3}} \leq \eta \|I_N\|_{L^2}^{6} + C_\eta.
%\]
%This allows absorption into existing \( L^2 \)-based estimates when higher-order norms (such as \( \|I_N\|_{L^2}^6 \)) are controlled, enabling closure of the energy inequality via Grönwall's lemma.

%Finally, by interpolation,
%\[
%\|I_N\|_{L^2}^{8/3} \leq C_\eta \|I_N\|_{L^2}^6 + C_\eta,
%\]
Finally, by this estimation, we completed the proof.
\end{proof}

\subsubsection{Summary of a priori estimate}

Collecting estimates and using boundedness of \(U\),
$$
\begin{array}{rcl}
\dfrac{d}{dt} \|I_N\|_{L^2}^2 + \|\nabla I_N\|_{L^2}^2 & \leq & C \Big(\|I_N\|_{L^2}^2 + \|I_N\|_{L^2}^3 + \|I_N\|_{L^2}^6 \\
& & + \|M_N\|_{L^2} \|I_N\|_{L^2} + \|M_N\|_{L^2} \|I_N\|_{L^2}^2 + \|M_N\|_{L^2}^4 \|I_N\|_{L^2}^2\Big).
\end{array}
$$

%Since
%\[
%\|M_N(t)\|_{L^2} \leq \|g\|_{L^1(0,T)} \sup_{s\in[0,t]} \|I_N(s)\|_{L^2},
%\]
This inequality is suitable for Grönwall arguments and uniform bounds on \(\|I_N\|_{L^2}\) and \(\|\nabla I_N\|_{L^2}\).

\subsubsection{Uniform \(L^2\)-bounds for the Galerkin Approximation}

We define
\[
Y(t) := \sup_{s \in [0,t]} \|I_N(s)\|_{L^2(\Omega)}^2,
\]
which is nondecreasing and finite since \(I_N\) is continuous in time.

From the energy inequality for the Galerkin solution \(I_N(t)\), we have for almost every \(t\),
\begin{equation}\label{eq:energy-inequality}
\begin{array}{rcl}
\dfrac{d}{dt} \|I_N(t)\|_{L^2}^2 + \|\nabla I_N(t)\|_{L^2}^2 & \leq & C \Big(\|I_N(t)\|_{L^2}^2 + \|I_N(t)\|_{L^2}^3 + \|I_N(t)\|_{L^2}^6  + \|M_N(t)\|_{L^2} \|I_N(t)\|_{L^2} \\
& & + \|M_N(t)\|_{L^2} \|I_N(t)\|_{L^2}^2 + \|M_N(t)\|_{L^2}^4 \|I_N(t)\|_{L^2}^2\Big).
\end{array}
\end{equation}
where 

\[
M_N(t) = \int_0^t g(t-s) I_N(s) \, ds.
\]

\begin{lemma}[Integral Inequality for \(Y(t)\)]\label{lem:intineqY}
The function \(Y:[0,T] \to \mathbb{R}_+\) defined by
\[
Y(t) := \sup_{s \in [0,t]} \|I_N(s)\|_{L^2(\Omega)}^2
\]
satisfies the integral inequality
\[
Y(t) \leq Y(0) + \int_0^t  C \Big(Y(s) + Y(s)^{3/2} + Y(s)^3\Big) ds, \quad \forall t \in [0,T].
\]
\end{lemma}
\begin{proof}

Using Minkowski's and Young's convolution inequalities, we estimate
\[
\|M_N(t)\|_{L^2} \leq \|g\|_{L^1(0,T)} \sup_{s \in [0,t]} \|I_N(s)\|_{L^2} = \|g\|_{L^1(0,T)} \sqrt{Y(t)}.
\]
Since $g$ is the log-normal distribution, then $ \|g\|_{L^1(0,T)}\leq 1$, therefore,
\[
\|M_N(t)\|_{L^2} \leq \sqrt{Y(t)}.
\]

Substituting into \eqref{eq:energy-inequality} yields
$$
\begin{array}{rcl}
\dfrac{d}{dt} \|I_N(t)\|_{L^2}^2 + \|\nabla I_N(t)\|_{L^2}^2 & \leq & C \Big(\|I_N(t)\|_{L^2}^2 + \|I_N(t)\|_{L^2}^3 + \|I_N(t)\|_{L^2}^6 \\
& &  + \|I_N(t)\|_{L^2}\sqrt{Y(t)} + \|I_N(t)\|_{L^2}^2\sqrt{Y(t)} + \|I_N(t)\|_{L^2}^2 Y(t)^2\Big).
\end{array}
$$

Since \(\|I_N(t)\|_{L^2}^2 \leq Y(t)\), we get
$$\dfrac{d}{dt} \|I_N(t)\|_{L^2}^2 \leq C \Big(Y(t) + Y(t)^{3/2} + Y(t)^3\Big).$$

Integrating over \([0,t]\) gives
\[
\|I_N(t)\|_{L^2}^2 \leq \|I_N(0)\|_{L^2}^2 + \int_0^t  C \Big(Y(s) + Y(s)^{3/2} + Y(s)^3\Big) ds.
\]

Using the definition of \(Y(t)\), it follows that
\[
Y(t) = \sup_{r \in [0,t]} \|I_N(r)\|_{L^2}^2 \leq Y(0) + \int_0^t  C \Big(Y(s) + Y(s)^{3/2} + Y(s)^3\Big) ds.
\]

The energy inequality \eqref{eq:energy-inequality} and the bounds on \(M_N\) stated above, together with the monotonicity of the supremum function complete the proof of the lemma.
\end{proof}

\medskip
To obtain a uniform (in time) bound for \(Y(t)\), we apply the following nonlinear integral inequality.

\begin{lemma}[Bihari's Inequality \cite{bihari}]\label{lem:Bihari}
Let \(Y:[0,T] \to \mathbb{R}_+\) be a continuous function satisfying
\[
Y(t) \leq Y_0 + \int_0^t f(Y(s)) \, ds,
\]
where \(f : [0,\infty) \to [0,\infty)\) is continuous and nondecreasing. Then
\[
F(Y(t)) \leq F(Y_0) + t,
\]
where \(F(y)\) is defined by
\[
F(y) := \int_{y_0}^y \frac{1}{f(s)} \, ds,
\]
for any \(y_0 > 0\) such that \(f(s) > 0\) for \(s \geq y_0\). Consequently,
\[
Y(t) \leq F^{-1}(F(Y_0) + t),
\]
for all \(t\) such that \(F(Y_0) + t < F(\infty)\).
\end{lemma}

By Lemma \ref{lem:Bihari}, the integral inequality implies that there exists \(T^* > 0\) depending on the initial data and the constant $C$, such that \(Y(t)\) remains bounded on \([0,T^*]\). 

In our case, the function
\[
f(y) :=  C \Big(y + y^{3/2} + y^3\Big)
\]
is continuous and strictly increasing on \([0,\infty)\), so the hypotheses of the lemma are satisfied.

We define
\[
F(y) := \int_{Y_0}^y \frac{1}{C( s + s^{3/2} + s^3)} \, ds.
\]
The function \(F:[Y_0,\infty)\rightarrow[0,F(\infty))\) is strictly increasing, hence \(F^{-1}\) is well-defined on \([0, F(\infty))\). Applying Bihari's inequality yields
\[
Y(t) \leq F^{-1}(t),
\]
which is finite for all \(t \in [0,T^*]\), where
\[
T^* := F(\infty) - F(Y_0) = \int_{Y_0}^\infty \frac{1}{C( s + s^{3/2} + s^3)} \, ds < \infty.
\]

\medskip

\noindent
%\textbf{Conclusion.} 
In conclusion, there exists a time \(T^* > 0\) such that \(Y(t)\) remains bounded on \([0,T^*]\). %Hence, we obtain uniform-in-\(N\) control of \(\|I_N(t)\|_{L^2(\Omega)}\) and \(\|\nabla I_N(t)\|_{L^2(\Omega)}\) on \([0,T^*]\). This is a crucial step for compactness arguments and for proving the existence of weak solutions.

This guarantees uniform-in-\(N\) control of
\[
\|I_N(t)\|_{L^2(\Omega)}^2 \leq Y(t) \leq C, \quad \text{for all } t \in [0,T^*].
\]

In particular, this estimate combined with the energy inequality also yields
\[
\|\nabla I_N(t)\|_{L^2}^2 \leq C, \quad \text{for all } t \in [0,T^*],
\]
providing uniform control of the spatial gradient.

This uniform boundedness is a key ingredient for passing to the limit as \(N \to \infty\) and proving existence of weak solutions.

\subsection{Passage to the limit in $N$ and existence of Weak Solution}
In light of the uniform a priori bounds, it follows from a classical compactness result; see for instance \cite{lions1969quelques} (Theorem 5.1 p.58) that the sequence $I_N$ has a strongly convergent subsequence, which we do not bother to relabel, in $L^2(\Omega_{T^*})$. So, there exists $I\in L^2(0,T^*;V)$ with $I_t\in L^2(0,T^*;V')$ such that:
\begin{itemize}[-]
    \item \(I_N\rightarrow I\) strongly in $L^2(\Omega_{T^*})$ and a.e. in $\Omega_{T^*}$.
    \item \(\nabla I_N \rightharpoonup \nabla I\) weakly in $L^2(\Omega_{T^*})$.
\end{itemize}
Moreover, the continuity of \( R(I_N):=(U-I_N)(I_N-M_N)\) and its bound
in \(L^2(\Omega_{T^*})\) which is a consequence of the inequalities obtained in Lemma \ref{lem:apriori}, along with a classical result; see \cite{lions1969quelques} Lemma 1.3 p 12, yield the weak convergence of $R(I_N)$ in $L^2(\Omega)$. Gathering all these results, we see that the limit $I$ verifies the weak formulation as defined in Definition \ref{def:weak}.\\
This ends the proof of Theorem \ref{thm:Existence of Weak solution}.\\
{{
\begin{proposition}
    If the initial state $I_0(x)$ verifies $0<I_0(x)<M$ for all $x\in\Omega$, then the weak solution $I(t,x)$ of System \eqref{eq:mainPDE} satisfies $$0\leq I(t,x)\leq M, \text{for all }t\in[0,T^*]\text{ and a.e. }x\in\Omega. $$ 
\end{proposition}
\begin{proof}We prove first that $I(t,x)\geq 0 $ for all $t\in[0,T^*]$ and a.e. $x\in\Omega$.
Let $I$ be the weak solution on $(0,T^*)\times\Omega$. We define the negative part $I^-(t,x) = \min(I(t,x), 0)$. Our goal is to show that the energy $E(t) = \|I^-(t)\|_{L^2}^2$ satisfies $E(t) = 0$ for all $t \in [0, T^*]$. Using $I^-$ as a test function in the weak formulation, we get:
$$\frac{1}{2} \frac{d}{dt} \|I^-(t)\|_{L^2}^2 + \|\nabla I^-(t)\|_{L^2}^2 = \int_\Omega \beta(U-I)(I - \mathcal{M}I) I^- \, dx$$
where $\mathcal{M}I=\int_0^tg(t-s)I(s,x)ds$. Using the same steps as in the proof of Lemma \ref{lem:apriori}, and noting that $II^-=(I^-)^2$ and $I^2I^-=(I^-)^3$, one obtains:
$$
\begin{array}{rcl}
\dfrac{d}{dt} \|I^-\|_{L^2}^2 + \|\nabla I^-\|_{L^2}^2 & \leq & C \Big(\|I^-\|_{L^2}^2 + \|I^-\|_{L^2}^3 + \|I^-\|_{L^2}^6 \\
& & + \|\mathcal{M}I\|_{L^2} \|I^-\|_{L^2} + \|\mathcal{M}I\|_{L^2} \|I^-\|_{L^2}^2 + \|\mathcal{M}I\|_{L^2}^4 \|I^-\|_{L^2}^2\Big).
\end{array}
$$
Define \[
Y(t) := \sup_{s \in [0,t]} \|I^-(s)\|_{L^2(\Omega)}^2.
\]
Then, proceeding as in the proof of Lemma \ref{lem:intineqY}, one finds:
\[
Y(t) = \sup_{r \in [0,t]} \|I^-(r)\|_{L^2}^2 \leq Y(0) + \int_0^t  C \Big(Y(s) + Y(s)^{3/2} + Y(s)^3\Big) ds.
\]
Knowing that $I_0(x)>0$, we have $Y(0)=0$. Hence,
\[
Y(t) \leq \int_0^t  f \Big(Y(s)\Big) ds,
\]
where \(f(y)=C(y+y^{3/2}+y^3)\). Now, since $\int_0^\epsilon \frac{1}{f(y)} dy \to \infty$, as $\epsilon\to0^+$, Bihari's inequality implies that $Y(t)=0$ for all $t\in[0,T^*]$. Therefore, $I^-(t,x)=0$ for all $t\in[0,T^*]$ and a.e. $x\in\Omega$.\\
To show that $I(t,x) \le M$ we also use an energy-based approach where the test function in the weak formulation is considered as $\psi(t,x) = (I(t,x) - M)^+=\max(I-M,0)$. Our goal is to prove that if $\psi(0) = 0$, then $\|\psi(t)\|_{L^2}^2 = 0$ for all $t\in[0,T^*]$. Testing the PDE with $\psi$, one gets:$$\frac{1}{2} \frac{d}{dt} \|\psi(t)\|_{L^2}^2 + \|\nabla \psi(t)\|_{L^2}^2 = \int_\Omega \beta(U-I)(I - \mathcal{M}I) \psi \, dx.$$
Now, repeating the same arguments as in the proof of $I^-(t,x)=0$, but using $$Y(t) = \sup_{s \in [0,t]} \|(I(s) - M)^+\|_{L^2(\Omega)}^2,$$ one can show that $\psi(t,x)=0$ for all $t\in[0,T^*]$ and a.e. $x\in\Omega$. 
\end{proof}}}

\section{Numerical Simulation}\label{sec:numerical}
In this section, we present some numerical results obtained from the models. Throughout this section, we use the rescaled transmission rate $\beta^\star=\beta/M$.
 
%\subsection{Discretization of ($*$)}
%We first set $\gamma = \beta^* \Bar{C}$, and
%$$f(t, I) = \gamma (M - I) \left(I - \int^t_0 g(t-s)I(s)ds\right)$$
%so that our equation becomes 
%$$I'(t) = f(t, I)$$
%and we use the trapezoidal rule to approximate the values of $I$: 
%$$I_{i+1} = I_i + \frac{\Delta t}{2} \left(f(t_i, I_i) + f(t_{i+1}, I_i + \Delta t f(t_i, I_i))\right)$$
%by setting $t_n = t_0 + n\Delta t \quad \forall n$.

%We also approximate the integral by the trapezoidal %rule. Let $$Q_i = \int^{t_i}_0 g(t-s)I(s)ds $$
%and
%$$q_k = g(t_i - t_k) I_k,\quad 0\leq k \leq i$$
%And so we approximate $Q_i$ as 
%\begin{align}
    %Q_i &=  \sum_{k=1}^i \left( \frac{\Delta t}{2} (q_k %+ q_{k-1})\right) \\
    %&= \frac{\Delta t}{2} \left(q_0 + 2(q_1 + \dots + %q_{i-1}) + q_i \right) \\
    %&= \frac{\Delta t}{2} \left(g(t_i)I_0 + 2 (g(t_i - %t_1)I_1 + \dots + g(t_i - t_{i-1})I_{i-1}) + g(0)I_i \right) \\
%\end{align}
%We can now compute the values of $I$ iteratively.

\subsection{Simulation of the Continuous Time Model}

\noindent\textbf{The integro-differential model \eqref{Syst1:IDE}}\\
To simulate the integro-differential system (4), we discretize the time interval using a uniform time step and approximate the integral term with a quadrature rule over the discrete time values. We apply the semi-implicit Euler method to evolve the system forward in time. This allows us to compute the evolution of the susceptible, infected, and recovered populations and to estimate the final infected population $I_{\infty}$ under various parameter values.

Figure \ref{fig:SIR_plot} shows the plot of the susceptible, infected, and recovered populations obtained by the system simulation. As expected and similar to the basic SIR model, the number of actively infected individuals increases, reaches a peak, and then decays back to zero. The cumulative number of infected individuals is nondecreasing and remains bounded by the total population number. 
\begin{figure}[H]
    \centering
    \includegraphics[width=0.9\linewidth]{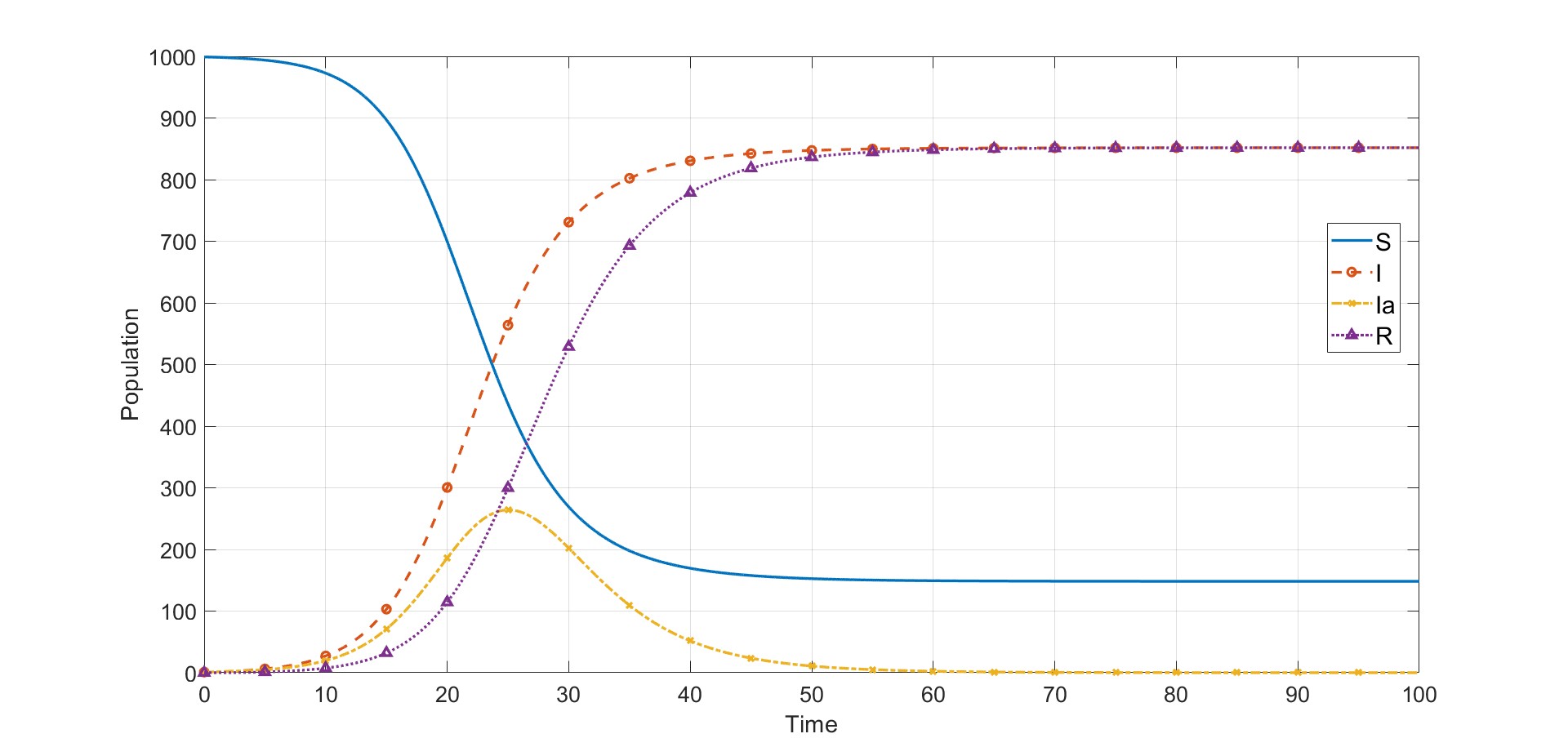}
    \caption{SIR Plot}
    \label{fig:SIR_plot}
\end{figure}

\noindent\textbf{Predicting $I_\infty$ numerically} \\
Figures \ref{fig:Iinf_beta} and \ref{fig:Iinf_I0} show the variation of $I_\infty$ as we vary the parameters $I_0$ and $\beta$. In both, we fix $m = 5.6$ and $st = 3.9$. Figure \ref{fig:Iinf_mu} shows the variation of $I_\infty$ as we vary the parameter $m$. Here, we fix $st = 3.9$. Figure \ref{fig:Iinf_st} shows the variation of $I_\infty$ as we vary the parameter $st$. Here, we fix $m = 5.6$. In both Figures \ref{fig:Iinf_mu} and \ref{fig:Iinf_st}, we fix $I_0 = 1$ and $\beta = 0.4$.\\
We compute the parameters of the log-normal distribution $g$ as 
$$\mu = ln\left(\frac{m^2}{\sqrt{st^2 + m^2}}\right)$$
$$\sigma = \sqrt{ln\left(1+\frac{st^2}{m^2}\right)}$$

We observe that the larger the infection rate $\beta$, the more people will be infected in the epidemic. Also, after exceeding approximately 0.7, the entire population is infected, regardless of the initial number of infections $I_0$. Similarly, the larger $I_0$ is, the larger the cumulative infected population $I_\infty$ will be. However, $I_0$ does not seem to play much of a role in determining whether the entire population is infected or not.

Moreover, as we increase the mean of the incubation period $m$, the total number of infected also increases. That is because the longer the incubation period, the more difficult it becomes for the infection to be contained. After some point around $m=14$, total infection is guaranteed.

Finally, as we vary the standard deviation of the incubation period between $7$ and $9$, we observe a slight decrease in $I_\infty$, followed by a quick increase to the total population after that.

\begin{figure}[H]
    \centering
    \begin{subfigure}[b]{0.45\textwidth}
        \includegraphics[width=\textwidth]{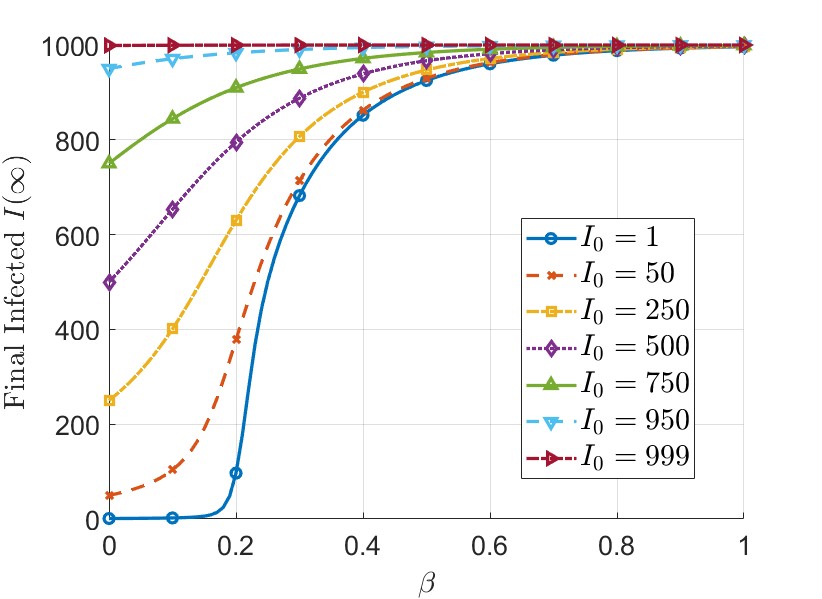}
        \caption{$I_\infty$ vs $\beta$ for different values of $I_0$}
        \label{fig:Iinf_beta}
    \end{subfigure}
    \hfill
    \begin{subfigure}[b]{0.45\textwidth}
        \includegraphics[width=\textwidth]{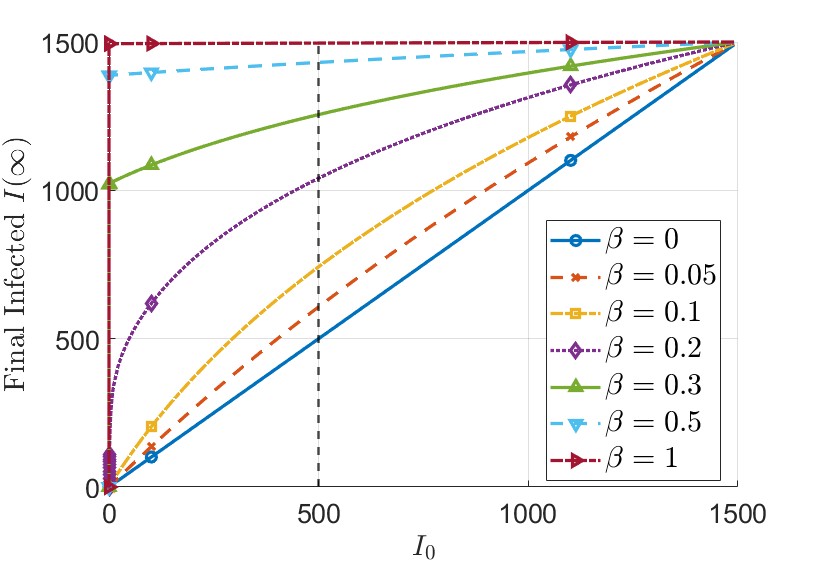}
        \caption{$I_\infty$ vs $I_0$ for different values of $\beta$}
        \label{fig:Iinf_I0}
    \end{subfigure}
    
    \vspace{0.5cm}
    
    \begin{subfigure}[b]{0.45\textwidth}
        \includegraphics[width=\textwidth]{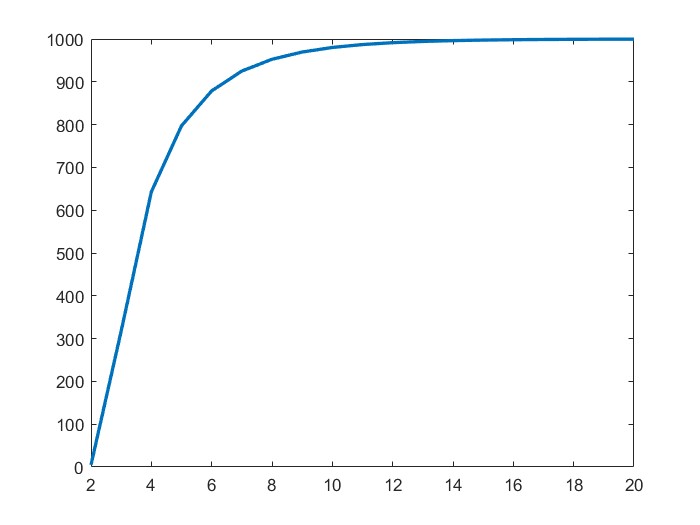}
        \caption{$I_\infty$ vs $m$}
        \label{fig:Iinf_mu}
    \end{subfigure}
    \hfill
    \begin{subfigure}[b]{0.45\textwidth}
        \includegraphics[width=\textwidth]{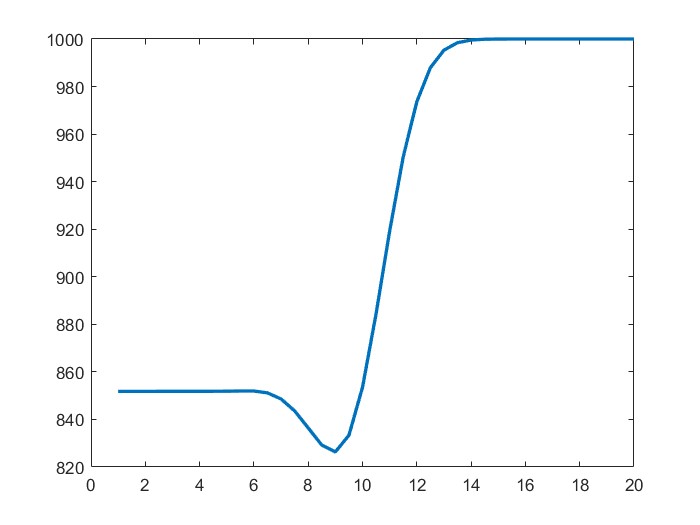}
        \caption{$I_\infty$ vs $st$}
        \label{fig:Iinf_st}
    \end{subfigure}
    
    \caption{Final infected population $I_\infty$ as a function of various parameters.}
    \label{fig:Iinf_combined}
\end{figure}

 \textbf{Phase Portrait $I_a$ vs $S$:} obtained numerically is shown in Figure \ref{phaseportrait}. Notice that the region to the right of the peaks is that in which an epidemic (or an increase in the active infected individuals) occurs while it does not occur to the left of it. We can see that, unlike the basic SIR model, the peaks do not trace a straight line.
\begin{figure}[H]
    \begin{center}
        \includegraphics[width=0.9\textwidth]{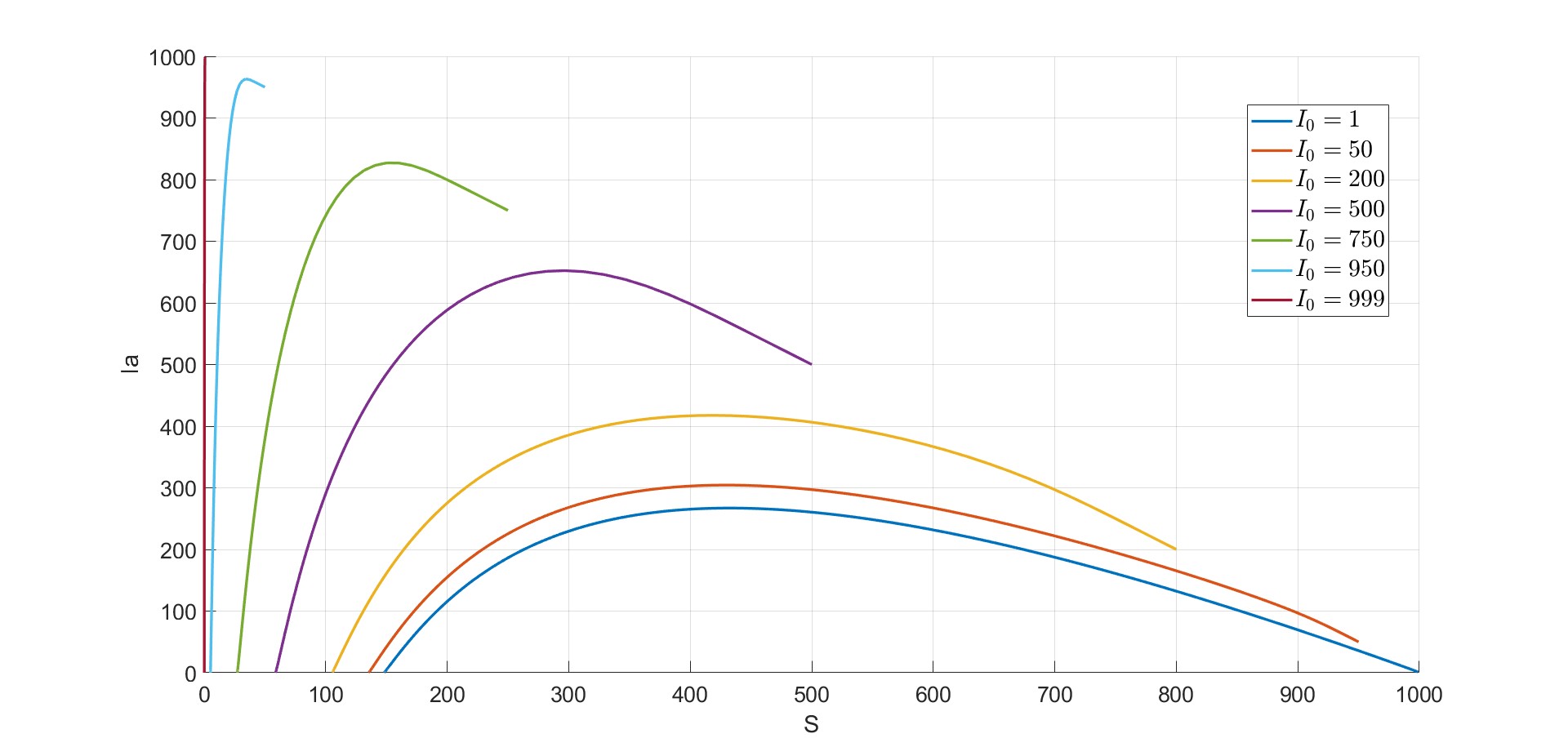}
    \end{center}
    \caption{$S$ vs $I_a$ for different values of $I_0$}
    \label{phaseportrait}
\end{figure}
\subsection{Simulation of the Spatio-Temporal Model }
In this section, we numerically solve the proposed spatio-temporal model \eqref{syst2:PDE}. The example presented is primarily for academic purposes and does not represent any reported data on an epidemic in Lebanon.\\
The domain $\Omega$ is obtained from a picture of the map of Lebanon. We employ a finite element discretization with a Crank-Nicolson time-stepping scheme for the heat equation to obtain the distribution of the total population $U(t,x)$, and a semi-implicit Euler time-stepping scheme to get the distribution of the infected population $I(t,x)$. Nonlinear
terms are solved explicitly, while linear differential operators are handled implicitly.\\
The domain $\Omega$ is discretized with an irregular mesh consisting of $9347$ triangles and $4873$ vertices (see Figure \ref{meshInit}) . The finite element method is implemented in FreeFem++, using $P_1$ elements for variables $U$, $I$, $S$, $I_a$ and $R$. We fix the time step $\Delta t=0.1$ and set the model parameters $\beta=0.005$ and diffusivity $\alpha=100$. The total population is assumed to be about $8.78423\times 10^6$ people concentrated in $11$ cities and towns with the highest population density in the capital ``Beirut''; it appears in red in Figure \ref{meshInit}. Figure \ref{SimFigures} shows the geographical evolution of the epidemic starting from an initial distribution of the infected population in the image to the top left of Figure \ref{SimFigures}. The initial infected population is assumed to be $0.05\%$ of the population in each city/town. Clearly, with the highest population density in the capital, there is a surge in the infected population in Beirut. A similar lower-amplitude increase in the number of infected people is observed in the north of Lebanon. This infection wave propagates to other regions with a delay in the regions with a lower population density. As expected theoretically, the total population stays constant at all times while the total active infected population initially increases to reach a peak and then starts declining again.

\begin{figure}[ht]
  \centering
\includegraphics[width=9cm]{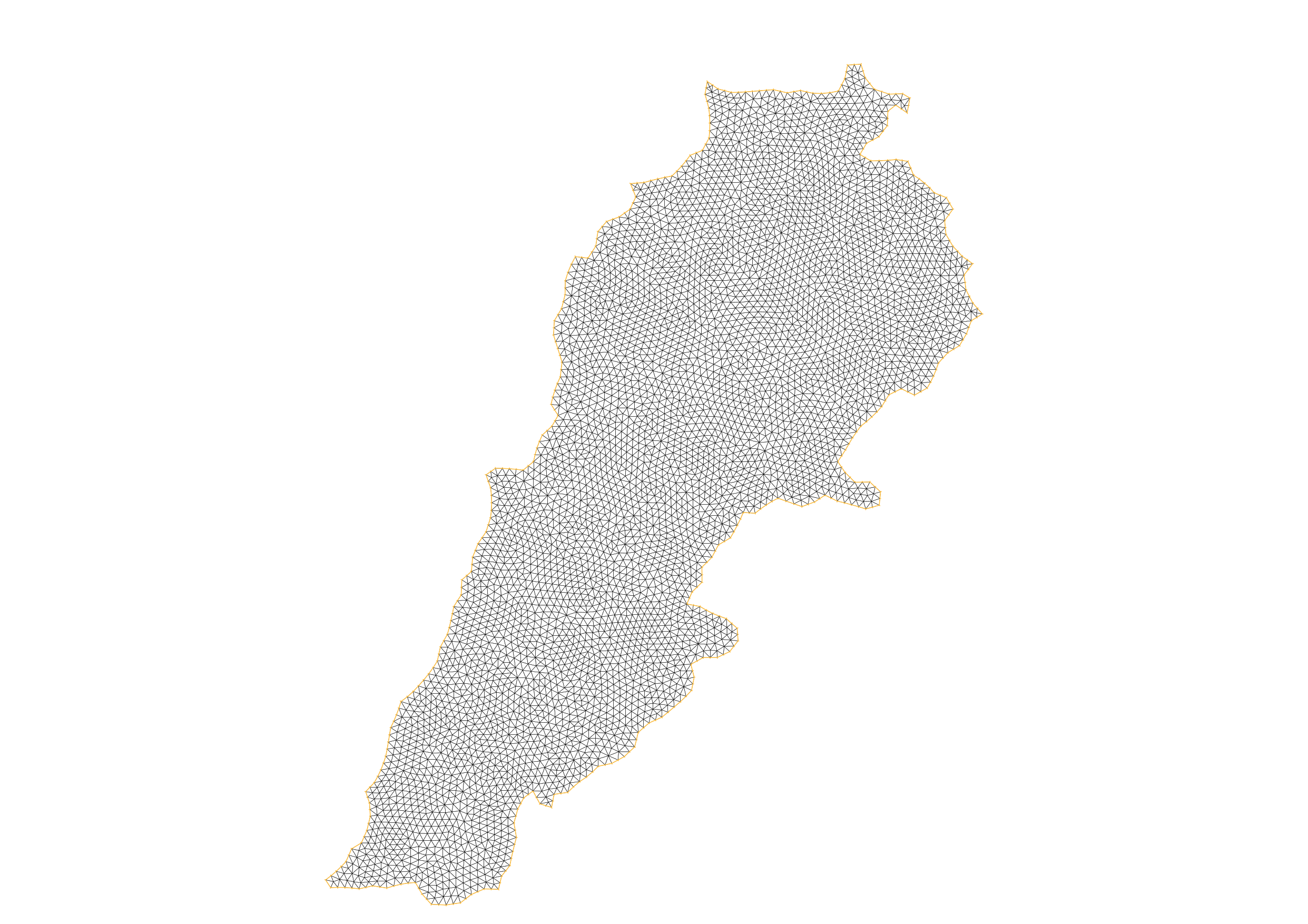}  
\includegraphics[width=4.5cm]{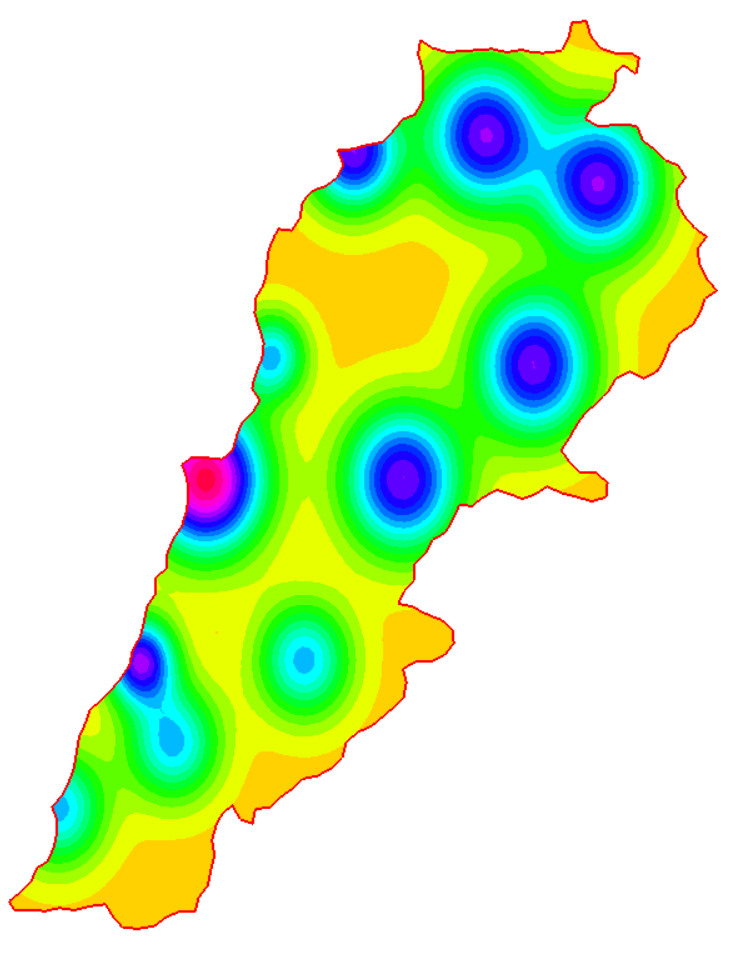}
\caption{Mesh of the map of Lebanon (left) and initial population distribution (right).}
\label{meshInit}
\end{figure}

\begin{figure}[ht]
\includegraphics[width=7cm]{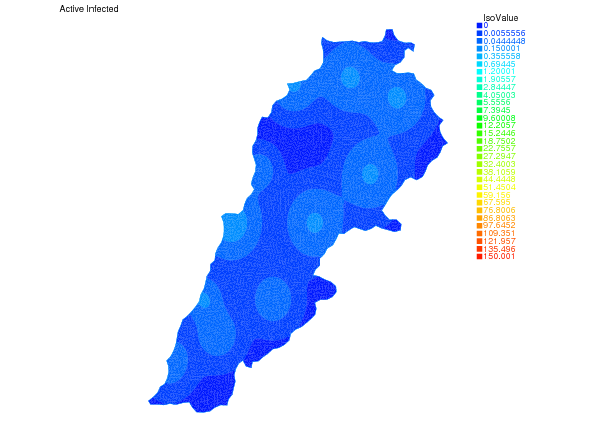}\includegraphics[width=7cm]{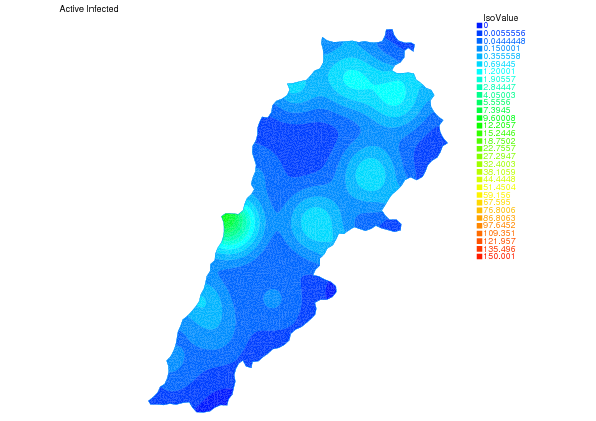}\\
\includegraphics[width=7cm]{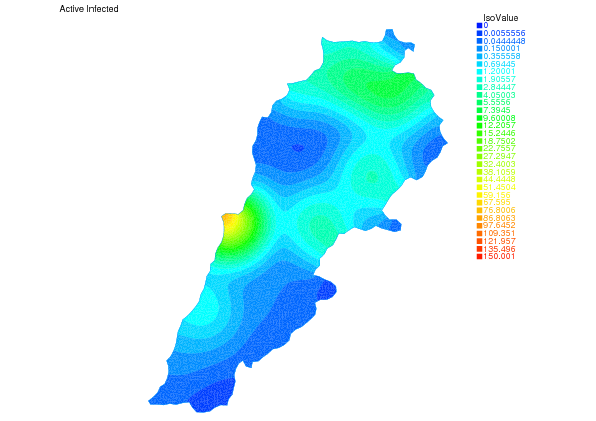}  \includegraphics[width=7cm]{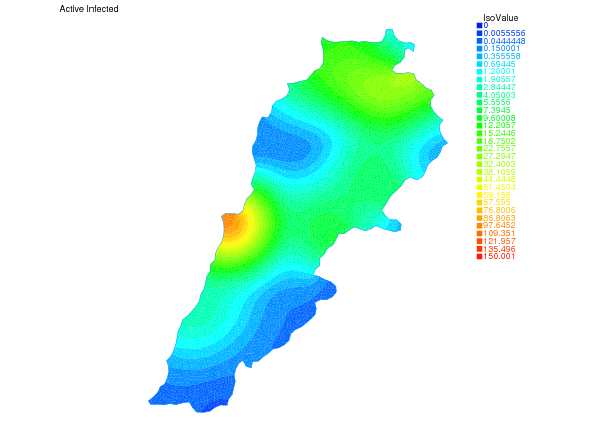}\\
\includegraphics[width=7cm]{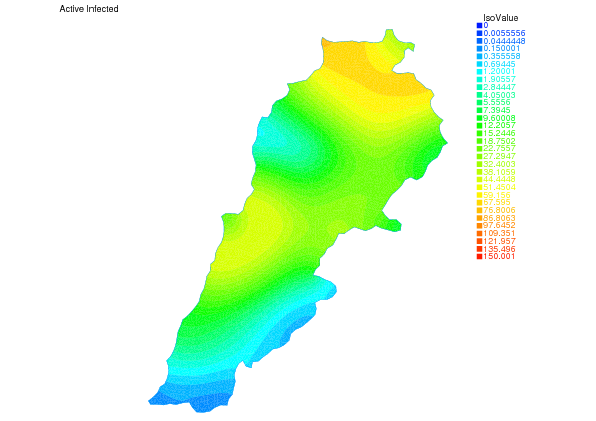}  \includegraphics[width=7cm]{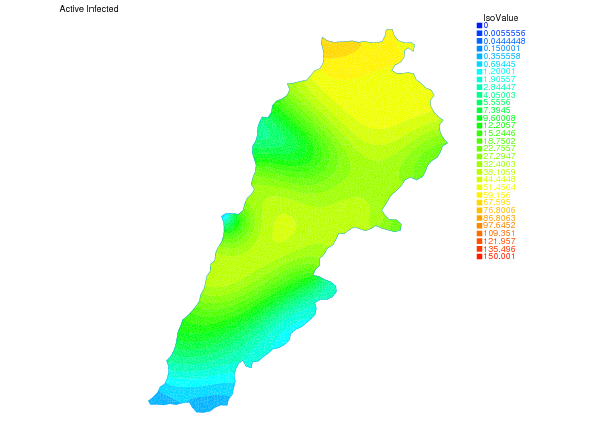}\\
\caption{Distribution of the active infected population $I_a$ at times $t=0.1$, $2.5$, $5$, $7.5$, $12$ and $14$ from top left corner to bottom right corner respectively.}
\label{SimFigures}
\end{figure}

\newpage

\section{Conclusion}\label{sec:Conc}

In conclusion, our work aims to expand the scope of previous models by implementing an integro-differential model designed to capture both the spatial dynamics of disease spread and the probabilistic dynamics of recovery, starting from a discrete framework and then taking the continuum limit. The paper also analyzed theoretical quantities that arise in the solutions of such equations, such as the threshold of epidemic, and the limiting value of infected people.

Moreover, the paper provides numerical approaches to model the system on a computer by discretization. In addition, a simulation was applied to a simplified scenario in Lebanon demonstrating the practical utility of our model, showing notable features like hot spots and infection waves.

Future directions could involve extending the model to include more complex movement terms than diffusion, utilizing a graph-theoretic framework, with nodes and edges, to model cities with discrete population centers, and a closed form of quantities arising in the solutions of the differential equation, like $I_\infty$, and so on.

\section*{Data Availability}
No new data was analyzed during this study.
\section*{Conflict of Interest}
The authors declare that they have no known competing financial interests or personal relationships that could have appeared to influence the work reported in this paper.

%\section*{Acknowledgment}
%This work was initiated during the Mathematics Summer Research Camp for undergraduate students at the American University of Beirut. The camp is supported by the Dean's office at the Faculty of Arts and Sciences and the Center for Advanced Mathematical Sciences (CAMS). The authors acknowledge the contribution of Miss Clara Riachi during the Summer Research Camp.
%\bibliographystyle{unsrt}
\bibliographystyle{abbrv}
\bibliography{biblioSpatio}
\end{document}